\documentclass [12pt,oneside]{amsart}
\usepackage{color}
\usepackage{amssymb,amsmath,amsthm,amsfonts}
\usepackage{enumerate}
\usepackage{setspace,bbm}
\usepackage{empheq}
\usepackage[all]{xy}
\usepackage{verbatim}
\usepackage[normalem]{ulem}
\usepackage[bookmarks,colorlinks,breaklinks]{hyperref}  
\hypersetup{linkcolor=blue,citecolor=blue,filecolor=blue,urlcolor=blue} 
\usepackage{todonotes}
\numberwithin{equation}{section}
\numberwithin{figure}{section}
\usepackage{mathtools}

\newtheorem{theorem}{Theorem}[section]
\newtheorem{proposition}[theorem]{Proposition}

\newtheorem{lemma}[theorem]{Lemma}{\Large }

\newtheorem{conjecture}[theorem]{Conjecture}

\theoremstyle{definition}
\newtheorem{definition}[theorem]{Definition}
\newtheorem{example}[theorem]{Example}
\newtheorem{remark}[theorem]{Remark}
\newtheorem{convention}[theorem]{Convention}

\definecolor{myblue}{rgb}{0.6, 0.9, 1}

\makeatletter

\newcommand{\Rmnum}[1]{\expandafter\@slowromancap\romannumeral #1@}
\makeatother

\definecolor{myblue}{rgb}{0.6, 0.9, 1}
\definecolor{mygreen}{rgb}{0,0,1}
\definecolor{purple}{rgb}{0.6,0.2,1}
\definecolor{orange}{rgb}{0.8,0,0.2}

\newcommand{\bC}{\mathbb{C}}
\newcommand{\bL}{\mathbb{L}}

\newcommand{\bZ}{\mathbb{Z}}
\newcommand{\bQ}{\mathbb{Q}}
\newcommand{\bP}{\mathbb{P}}
\newcommand{\R}{\mathbb{R}}

\newcommand{\A}{\mathbb{A}}

\newcommand{\ord}{\operatorname{ord}}
\newcommand{\del}{\partial}
\newcommand{\eps}{\varepsilon}

\newcommand{\Res}{\operatorname{Res}}

\newcommand{\Gal}{\operatorname{Gal}}
\newcommand{\kbar}{\overline{k}}
\newcommand{\Qbar}{\overline{\bQ}}
\newcommand{\Kbar}{\overline{K}}

\newcommand{\supp}{\operatorname{supp}}

\newcommand{\cS}{\mathcal{S}}

\begin{document}
\title[Variation of heights for Fatou points]{Variation of canonical height for Fatou points on $\bP^1$}

\author{Laura DeMarco} 
\address{Department of Mathematics, Harvard University, 1 Oxford Street, Cambridge, MA 02138, USA}
\email{\href{mailto:demarco@math.harvard.edu}{demarco@math.harvard.edu}}

\author{Niki Myrto Mavraki}
\address{Department of Mathematics, Harvard University, 1 Oxford Street, Cambridge, MA 02138, USA}
\email{\href{mailto:mavraki@math.harvard.edu}{mavraki@math.harvard.edu}}
\date{October 4, 2022}


\begin{abstract}
Let $f: \bP^1\to \bP^1$ be a map of degree $>1$ defined over a function field $k = K(X)$, where $K$ is a number field and $X$ is a projective curve over $K$.   For each point $a \in \bP^1(k)$ satisfying a dynamical stability condition, we prove that the Call-Silverman canonical height for specialization $f_t$ at point $a_t$, for $t \in X(\Qbar)$ outside a finite set, induces a Weil height on the curve $X$;  i.e., we prove the existence of a $\bQ$-divisor $D = D_{f,a}$ on $X$ so that the function $t\mapsto \hat{h}_{f_t}(a_t) - h_D(t)$ is bounded on $X(\Qbar)$ for any choice of Weil height associated to $D$.  We also prove a local version, that the local canonical heights $t\mapsto \hat{\lambda}_{f_t, v}(a_t)$ differ from a Weil function for $D$ by a continuous function on $X(\bC_v)$, at each place $v$ of the number field $K$.  These results were known for polynomial maps $f$ and all points $a \in \bP^1(k)$ without the stability hypothesis \cite{Ingram:polyvariation, Favre:Gauthier:continuity}, and for maps $f$ that are quotients of endomorphisms of elliptic curves $E$ over $k$ and all points $a \in \bP^1(k)$ \cite{Tate:variation, Silverman:VCHII}.  Finally, we characterize our stability condition in terms of the geometry of the induced map $\tilde{f}: X\times \bP^1 \dashrightarrow X\times \bP^1$ over $K$; and we prove the existence of relative N\'eron models for the pair $(f,a)$, when $a$ is a Fatou point at a place $\gamma$ of $k$, where the local canonical height $\hat{\lambda}_{f,\gamma}(a)$ can be computed as an intersection number.  
\end{abstract}
\maketitle

\thispagestyle{empty}

\bigskip
\section{Introduction}

In this article, we study the variation of canonical height in families of maps $f: \bP^1 \to \bP^1$.  More precisely, we fix a number field $K$ and a smooth projective curve $X$ defined over $K$. Let $k = K(X)$ be the associated function field, and let $\Kbar$ denote an algebraic closure of $K$.  Any map $f: \bP^1\to \bP^1$ of degree $d$ defined over $k$ will specialize to a morphism $f_t : \bP^1\to \bP^1$ of degree $d$, defined over $\Kbar$, for all but finitely many $t \in X(\Kbar)$.  For points $a \in \bP^1(k)$, we are interested in properties of the function $t \mapsto \hat{h}_{f_t}(a_t)$, where $\hat{h}_{f_t}$ is the Call-Silverman canonical height for $f_t$ as defined in \cite{Call:Silverman}, as $t$ varies in $X(\Kbar)$.

An important case was studied in the early 1980s.  Given any elliptic surface $E \to X$ with a zero section, defined over a number field $K$, and given a section $P: X\to E$ also defined over $K$, the fiber-wise canonical height $t\mapsto \hat{h}_{E_t}(P_t)$ is known to define a Weil height on the base curve $X(\Kbar)$ \cite{Tate:variation}.  That is, there exists a $\bQ$-divisor $D_{E,P}$ on $X$, of degree equal to the geometric canonical height $\hat{h}_E(P)$ (viewing $E$ as an elliptic curve over the function field $k$) so that 
\begin{equation}  \label{Tate VCH}
	\hat{h}_{E_t}(P_t) - h_{D_{E,P}}(t) = O(1)
\end{equation}
for any choice of Weil height associated to $D_{E,P}$.  The notation $O(1)$ represents a bounded function, defined on the complement of finitely many points in $X(\Kbar)$; the bound depends on the pair $(E,P)$ and the choice of Weil height $h_{D_{E,P}}$.  This can be viewed as a dynamical example on $\bP^1$ as follows.  Projecting each smooth fiber $E_t$ to $\bP^1$ by the natural degree-two quotient that identifies a point $x \in E_t$ with its inverse $-x$, and taking, for example, the multiplication-by-2 endomorphism on $E_t$, we obtain a family of maps $f_t: \bP^1 \to \bP^1$, well-defined for all but finitely many $t \in X(\Kbar)$.  See, for example, \cite[\S6.4]{Silverman:dynamics}.  The section $P$ projects to an element $p \in \bP^1(k)$, and we have $\hat{h}_{f_t}(p_t) = 2 \, \hat{h}_{E_t}(P_t)$, so that 
\begin{equation} \label{Tate dyn VCH}
	\hat{h}_{f_t}(p_t) -  h_{D_{f,p}}(t) = O(1)
\end{equation}
on the complement of finitely many points in $X(\Kbar)$, for a $\bQ$-divisor $D_{f,p} = 2\, D_{E,P}$ on $X$.  

For any given map $f: \bP^1\to \bP^1$ defined over $k$ of degree $ > 1$, and each point $a \in \bP^1(k)$, Call and Silverman proved that the specializations satisfy
\begin{equation} \label{CS bound}
	\hat{h}_{f_t}(a_t) - h_D(t) = o(h_D(t))
\end{equation}
as $h_D(t) \to \infty$, for any choice of Weil height $h_D$ on $X(\Kbar)$ associated to a divisor $D$ of degree equal to the geometric (i.e., over $k$) canonical height $\hat{h}_f(a)$ \cite[Theorem 4.1]{Call:Silverman}.   Recently, Ingram improved the error term $o(h_{D}(t))$ in \eqref{CS bound} to $O(h^{2/3}_{D}(t))$ \cite{Ingram:specialization}. Inspired by \eqref{Tate dyn VCH} and \eqref{CS bound}, Call and Silverman asked if there can exist a divisor $D = D_{f,a}$ on $X$ so that the stronger result of the form \eqref{Tate dyn VCH} will hold for every $f$ and $a$; see the Remark after Theorem 4.1 in \cite{Call:Silverman}. We give a partial answer to this question.

\begin{definition}
A point $a \in \bP^1(k)$ is said to be {\bf totally Fatou} for $f$ if it is an element of the non-archimedean Fatou set at every place $\gamma \in X(\Kbar)$ of $k$. 
\end{definition}

We refer the reader to Section \ref{Fatou section} for more information. 
We note here that throughout this article we identify the places of $k$ with those of $k\otimes \overline{K}$ and with the points $\gamma\in X(\overline{K})$. The notion of a totally Fatou point has also appeared in \cite{Petsche} in the setting of number fields.

\begin{theorem}  \label{dyn VCH}
Let $K$ be a number field and $X$ a smooth projective curve over $K$.  Let $f : \bP^1\to \bP^1$ be a map of degree $>1$ defined over $k = K(X)$, and suppose that $a \in \bP^1(k)$ is totally Fatou for $f$.  Then there exists a $\bQ$-divisor $D = D_{f,a}$ on $X$, of degree equal to the geometric height $\hat{h}_f(a)$, so that $t\mapsto \hat{h}_{f_t}(a_t)$ defines a Weil height for $D$ on $X(\Kbar)$.  More precisely, for any choice of Weil height $h_D$ associated to $D$, we have 
	$$\hat{h}_{f_t}(a_t) - h_D(t) = O(1)$$
as a function of $t\in X(\Kbar)\setminus Y$ for a finite set $Y$, where $\hat{h}_{f_t}(a_t)$ is well defined. The bounds on $\hat{h}_{f_t}(a_t) - h_D(t)$ depend on $f$, $a$, and the choice of Weil height $h_D$.  
\end{theorem}

We shall see that the divisor $D$ is given by 
\begin{equation} \label{divisor def}
	D_{f,a} = \sum_{\gamma \in X(\Kbar)} \hat{\lambda}_{f,\gamma}(a) \, \cdot \gamma 
\end{equation}
where $\hat{h}_f = \sum_\gamma \hat{\lambda}_{f,\gamma}$ is a local decomposition of the geometric canonical height for $f$ over $k$.  The fact that $D$ is a $\bQ$-divisor for totally Fatou points $a$, so that $\hat{\lambda}_{f,\gamma}(a) \in \bQ$ and therefore also $\hat{h}_f(a) \in \bQ$, is new; see Proposition \ref{rational}, addressing a question in \cite{DG:rationality}.  As a special case of Theorem \ref{dyn VCH} we recover \eqref{Tate VCH} and \eqref{Tate dyn VCH}, because all points in $\bP^1(k)$ are totally Fatou for the maps $f$ coming from elliptic curves.

The statement of Theorem \ref{dyn VCH} was proved by Ingram for polynomial maps $f(z) \in k[z]$ and for all points $a \in \bP^1(k)$ without the totally Fatou assumption \cite{Ingram:polyvariation}.  Polynomial maps have a totally invariant super-attracting fixed point at $\infty$, simplifying computations of the canonical height.   In fact, much more is known for polynomials $f$ and for maps $f$ coming from elliptic curves, and we address some of this  below in the context of Theorem \ref{local dyn VCH}; see the works of Favre and Gauthier \cite{Favre:Gauthier:continuity,FG:book} and of Silverman \cite{Silverman:VCHI,Silverman:VCHII,Silverman:VCHIII}.  However, even with the totally Fatou assumption, new complications arise for rational maps that do not exist for polynomials or maps coming from elliptic curves, as we discuss after Theorem \ref{local dyn VCH} and illustrate by example in Section \ref{example section}.

\medskip\noindent{\bf The totally Fatou condition.}
In contrast with the setting of number fields, it may be true that {\em every} point $a\in \bP^1(k)$ is either preperiodic or totally Fatou for maps $f$ defined over $k$.  (Note that the statement of Theorem \ref{dyn VCH} holds trivially when $a$ is preperiodic for $f$, as $\hat{h}_{f_t}(a_t) = 0$ at all points $t$ where $f_t$ is defined, and we can take $D = 0$.)  We know of no examples, nor any mechanisms to prove existence, of maps $f$ defined over $k$ and points $a \in \bP^1(k)$ with infinite orbit for which $a$ lies in the non-archimedean Julia set of $f$ at a place $\gamma$ of $k$. 

\begin{conjecture} \label{conj}
Let $K$ be a number field and $X$ a smooth projective curve over $K$.  Let $f: \bP^1\to \bP^1$ of degree $>1$ be defined over $k=K(X)$.  Then every point $a\in\bP^1(k)$ is either preperiodic or totally Fatou for $f$.  
\end{conjecture}

Note that the conjecture remains open for polynomial maps $f$, though the conclusion of Theorem \ref{dyn VCH} is known to hold for all points $a \in \bP^1(k)$ in that case \cite{Ingram:polyvariation}.  In Section \ref{example section}, we observe that for all of the previously known cases of Theorem \ref{dyn VCH} in the literature where the maps $f$ are {\em not} polynomials (nor conjugate to polynomials), the points $a \in \bP^1(k)$ are totally Fatou for $f$.

Here we prove that ``most" points in $\bP^1(k)$, from a density point of view, are totally Fatou.  Let $k_\gamma$ denote the completion of $k$ at the place $\gamma \in X(\Kbar)$.

\begin{theorem} \label{dense}
For any $f: \bP^1 \to \bP^1$ of degree $>1$ defined over $k = K(X)$, the set of totally Fatou points for $f$ in $\bP^1(k)$ is open and dense in the product topology on $\bP^1(k)$, coming from the embedding of $k$ into $\prod_{\gamma \in X(\Kbar)} k_\gamma$.  
\end{theorem}
 
\noindent
Theorem \ref{dense} exploits the non-local-compactness of $k_\gamma$; it is false for maps $f$ defined over number fields $K$, where the Fatou set in a completion $K_v$ can fail to be dense at archimedean or non-archimedean places $v$.

To understand the totally Fatou condition better, we relate it to the geometry of the induced rational map 
	$$\tilde{f}:  X\times \bP^1 \dashrightarrow  X \times \bP^1$$ 
on the complex surface $X \times \bP^1$, defined by $(t,z) \mapsto (t, f_t(z))$. 
 Let $I(\tilde{f})$ denote the (finite) indeterminacy set of $\tilde{f}$ in $(X\times \bP^1)(\Kbar)$.  For a point $a \in \bP^1(k)$, let $C_a$ denote the graph in $X\times \bP^1$ of the associated holomorphic map $t \mapsto a(t)$ from $X$ to $\bP^1$.
 
\begin{theorem} \label{modification}
Let $f: \bP^1 \to \bP^1$ be of degree $ > 1$, defined over a function field $k = K(X)$, with the number field $K$ chosen so that all indeterminacy points of $\tilde{f}$ lie in $(X \times \bP^1)(K)$.  A point $a \in \bP^1(k)$ is totally Fatou for $f$ if and only if there exists a birational morphism $Y \to  X \times \bP^1$, defined over $K$, so that the induced map 
$$\xymatrix{ Y \ar[d] \ar@{-->}[r]^{\tilde{f}_Y} & Y\ar[d] \\   X\times\bP^1  \ar@{-->}[r]^{\tilde{f}} & X\times\bP^1  }$$	
 satisfies $C_{f^n(a)}^Y \cap I (\tilde{f}_Y ) = \emptyset$ for all $n\geq 0$, where $C_{f^n(a)}^Y$ is the proper transform of the curve $C_{f^n(a)}$ in $Y$.   Moreover, the modification $Y$ can be chosen so that $\tilde{f}_Y$ is algebraically stable, meaning that no curve is mapped by an iterate $(\tilde{f}_Y)^n$ into the indeterminacy set $I(\tilde{f}_Y)$, and such that $C_{f^n(a)}^Y$ intersects the singular fibers of the projection $Y \to X$ only at smooth points, for all $n\ge 0$.
\end{theorem}  

\begin{remark}
It was proved in \cite[Theorem E]{DF:degenerations2} that, for every $f$ of degree $>1$ over $k$, there exists a modification $Y \to X\times \bP^1$ so that the induced map $\tilde{f}_Y: Y \dashrightarrow Y$ is algebraically stable.  Theorem \ref{modification} implies we can further modify $Y$ so that the orbit of $C_a$ is disjoint from the indeterminacy locus of $\tilde{f}_Y$, when $a$ is totally Fatou.  The choice of $Y$ will depend on $a$.
\end{remark} 

We use Theorem \ref{modification} to prove that the geometric local canonical height $\hat{\lambda}_{f,\gamma}(a)$ can be computed as an intersection number in $Y$, assuming the point $a$ is Fatou at $\gamma$; see Theorem \ref{intersection computation} and compare with \cite[Theorem 6.1]{Call:Silverman}.  In analogy with the study of elliptic curves and abelian varieties, the concept of a ``weak N\'eron model" at a place $\gamma$ of $k$ was introduced in \cite{Call:Silverman} for dynamical systems; but, it is known that these models often fail to exist for maps $f: \bP^1\to \bP^1$ defined over $k_\gamma$ in the absence of good reduction, for example when there is a repelling periodic point in $k_\gamma$ \cite{Hsia:weak}. In fact, as Ingram noted in the Introduction to \cite{Ingram:polyvariation}, if $f: \bP^1\to \bP^1$ defined over $k$ is neither Latt\`es nor isotrivial, then it cannot have a weak N\'eron model at every place $\gamma$.  The proof of Theorem \ref{modification} provides the existence of a {\em relative} type of weak N\'eron model, for a pair $(f,a)$ with $a$ being Fatou at $\gamma$, in which the orbit of the Fatou point can be arranged to be integral.

Theorem \ref{modification} follows from the proof of \cite[Theorem D]{DF:degenerations2} and the classification of $\gamma$-adic Fatou components in the Berkovich projective line $\bP^{1,an}_\gamma$ (over a complete and algebraically closed field $\bC_\gamma$ containing the completion $k_\gamma$) \cite{RiveraLetelier:Asterisque} \cite{Benedetto:wandering} \cite[Appendix]{DF:degenerations2}; many of the ideas were already present in \cite{Hsia:weak}, and what remained was to show that the full orbit $\{f^n(a)\}_{n\geq 0}$ can be disjoint from the indeterminacy set after only finitely many blowups of $X\times \bP^1$.

\medskip\noindent{\bf Local version of Theorem \ref{dyn VCH}.}
In the setting of elliptic surfaces $E\to X$, Silverman strengthened Tate's result \eqref{Tate VCH} by showing that the function $B_{E,P}(t) := \hat{h}_{E_t}(P_t) - h_{D_{E,P}}(t)$, defined for all but finitely many $t\in X(\Qbar)$, can be expressed as a sum over all places $v$ of the number field $K$ of functions with good behavior \cite{Silverman:VCHI, Silverman:VCHII, Silverman:VCHIII}.  More precisely, he proved that the local height functions for $\hat{h}_{E_t}$ on $E_t(\Qbar)$ and for $h_{D_{E,P}}$ on $X(\Qbar)$ can be chosen so that all $v$-adic contributions to $B_{E,P}$ extend to define {\em continuous} functions on $X(\bC_v)$, even across the singular fibers, and that {\em all but finitely many} of the $v$-adic contributions are $\equiv 0$.

We also prove a local continuity result, strengthening the conclusion of Theorem \ref{dyn VCH}:

\begin{theorem} \label{local dyn VCH}
Under the hypotheses of Theorem \ref{dyn VCH}, we assume that the number field $K$ is extended so that $\supp D_{f,a} \subset X(K)$.  There are local decompositions 
	$$\hat{h}_{f_t}(a_t) = \frac{1}{[K:\bQ]} \frac{1}{|\Gal(\Kbar/K)\cdot t|} \sum_{x \in \Gal(\Kbar/K)\cdot t} \sum_{v \in M_K} N_v \, \hat{\lambda}_{f_x,v}(a_x)$$
and
	$$h_D(t) = \frac{1}{[K:\bQ]} \frac{1}{|\Gal(\Kbar/K)\cdot t|} \sum_{x\in \Gal(\Kbar/K)\cdot t} \sum_{v \in M_K} N_v \, \lambda_{D,v}(x)$$
for $t \in X(\Kbar)\setminus \supp D_{f,a}$ so that the function 
	$$V_v(t) := \hat{\lambda}_{f_t,v}(a_t) - \lambda_{D,v}(t)$$
extends to a continuous function on the Berkovich analytification $X^{an}_{\bC_v}$ for each place $v$ of $K$.  
\end{theorem}

\noindent
Here, $M_K$ denotes the set of places of the number field $K$, and the weights $N_v = [K_v: \bQ_v]$ are the same as those appearing in the product formula $1 = \prod_{v \in M_K} |\alpha|_v^{N_v}$ for $\alpha \in K^*$.   
The conclusion of Theorem \ref{local dyn VCH} is known for polynomial maps $f(z) \in k[z]$ and for all $a\in\bP^1(k)$ without the totally Fatou hypothesis \cite{Ghioca:Ye:cubics, Favre:Gauthier:continuity}; their proofs take clever advantage of the compactness of the orbit-closures of all points in the $\gamma$-adic Julia sets, as subsets of $\bP^1(k_\gamma)$ (see \cite[Theorem 3]{Favre:Gauthier:continuity}, \cite[Theorem 4.8]{Hsia:weak}, \cite[Proposition 6.7]{Trucco}), which does {\em not} hold for general rational maps $f$.  See, for example, the $f$ of \S\ref{divergent example}.  Moreover, even for totally Fatou points $a \in \bP^1(k)$, the proof of Theorem \ref{local dyn VCH} requires a new approach.  The local canonical height functions $\hat{\lambda}_{f_t,v}$ for polynomials $f$ can be normalized so they are always non-negative.  The challenge here is the absence of a uniform lower bound on the functions $V_v$ of Theorem \ref{local dyn VCH}, independent of $a$.  (This unboundedness was exploited in \cite{DO:discontinuity} to show $V_v$ can fail to extend continuously for maps $f(z) \in k(z)$ when a point $a \in \bP^1(k')$ is defined over a larger field such as $k' = K_v(X)$; see Remark \ref{DO}.)  

Finally, we remark that Theorem \ref{local dyn VCH} as stated does not imply Theorem \ref{dyn VCH}.  For polynomial maps $f$ and each $a \in \bP^1(k)$, the functions $V_v$ of Theorem \ref{local dyn VCH} will satisfy $V_v \equiv 0$ at all but finitely many places $v$ of $K$ \cite{Ingram:polyvariation}, as is the case for sections of elliptic surfaces \cite{Silverman:VCHIII}.  However, by contrast, it is {\em not} the case that the functions $V_v$ will be $\equiv 0$ for all but finitely many places $v$ for general rational maps $f$; there can be nontrivial contributions at infinitely many places, even for totally Fatou points $a \in \bP^1(k)$.  See the example of \S\ref{all primes}; such examples were studied in depth in \cite{Mavraki:Ye}.  Nevertheless, we extract the summability of the magnitudes of $V_v$, over all places $v$ of $K$, from the proof of Theorem \ref{local dyn VCH}.

\medskip\noindent{\bf Julia points.}
We use the totally Fatou hypothesis on $a \in \bP^1(k)$ in a crucial way in our proofs of Theorems \ref{dyn VCH} and \ref{local dyn VCH}.  
The study of Julia points with infinite orbit is more subtle.  As we show in Section \ref{example section}, there exist examples of the following:\begin{enumerate}
\item  a map $f: \bP^1 \to \bP^1$ of degree 2 defined over $k = \bQ(t)$ with bad reduction at $t=0$, for which the non-archimedean Julia set at $t=0$ is a Cantor set in the completion $\bP^1(k_0)$ at $t=0$, and the local geometric height $\hat{\lambda}_{f,0}(a)$ is in $\R\setminus \bQ$ for {\em all} Julia points $a$ with infinite orbit. See \S\ref{Julia example 1}; compare the main results of \cite{DG:rationality}.
\item  a map $f: \bP^1 \to \bP^1$ of degree 2 defined over $k = \bQ(t)$ with bad reduction at $t=0$, and a point $a$ defined by a formal power series in $\bQ[[t]]$ in the non-archimedean Julia set of $f$ at $t=0$ for which, at the place $v = \infty$ of $\bQ$, the function $V_v$ of Theorem \ref{local dyn VCH} will fail to be defined at $t=0$.  See \S\ref{divergent example}.
\end{enumerate}

For either example, if such a point $a$ can be constructed to be {\em algebraic} over $k$, then, upon replacing $k$ with a finite extension, it would provide a counterexample to Conjecture \ref{conj}, and the results we prove for totally Fatou points would fail to extend to all $a \in \bP^1(k)$.  More precisely, example (1) would show that the divisor $D$ constructed in Theorem \ref{dyn VCH}, defined by \eqref{divisor def}, needs to be an $\R$-divisor instead of a $\bQ$-divisor; compare Proposition \ref{rational}.  Example (2) would show that the sequences of functions converging to define the $V_v$ of Theorem \ref{local dyn VCH} would not always converge uniformly in the neighborhood of a singularity; compare Theorem \ref{uniform convergence}.

\begin{remark}  \label{DO} It is known that, working with maps $f$ defined over the field $\ell = \bC(t)$, there exist points $a \in \bP^1(\ell)$ that are totally Fatou for $f$ but for which the analog of the (archimedean) function $V_\infty$ of Theorem \ref{local dyn VCH} is unbounded on the base curve $\bP^1(\bC)$ \cite{DO:discontinuity}.  The construction in \cite{DO:discontinuity} is different from the construction for example (2) and uses Baire Category.  The results of Favre and Gauthier show that such examples over $\ell$ or examples of the types (1) and (2) above cannot exist for polynomials $f$ \cite{Favre:Gauthier:continuity}. 
\end{remark}

\medskip\noindent{\bf Acknowledgements.}
We thank Hexi Ye for helpful discussions about this problem.  We also thank the anonymous referees for their comments and suggestions.  This research was supported by NSF grant DMS-2050037.

\bigskip
\section{$M_K$-terminology}

In this section, we fix some basic terminology associated to the number field $K$ and remind the reader of fundamental facts about elements of $k = K(X)$.

Let $M_K$ denote the set of places of the number field $K$, each giving rise to an absolute value $|\cdot|_v$ on $K$ which is normalized to extend one of the standard absolute values on the field $\bQ$ of rational numbers.  The set $M_K$ satisfies the product formula, 
	$$\prod_{v \in M_K} |x|_v^{N_v} = 1$$
for all $x \in K^*$.   We let $K_v$ denote the completion of $K$ at $v$, so that 
	$$N_v = [K_v: \bQ_v].$$
For each place $v$ of $K$, we let $\bC_v$ be the completion of an algebraic closure $\overline{K_v}$.  We also fix an embedding $\Kbar \hookrightarrow \bC_v$.  We let $X^{an}_v$ denote the Berkovich analytification of the curve $X$ over the field $\bC_v$.

We will use the following terminology, as in \cite[Chapter 10]{Lang:Diophantine}:  An {\bf $M_K$-constant} is a function $\mathfrak{C}: M_K \to \R$ so that $\mathfrak{C}_v = 0$ at all but finitely many places $v$.  An {\bf $M_K$-quasiconstant} is a function $\mathfrak{C}: M_K \to \R$ such that 
	$$\sum_{v\in M_K} N_v \left|\mathfrak{C}_{v}\right| < \infty.$$
A collection of functions $\mathfrak{f}_v: Y \to \R$, for $v\in M_K$, defined on a set $Y$, is {\bf $M_K$-bounded} if there exists an $M_K$-constant $\mathfrak C$ so that $|\mathfrak{f}_v(y)| \leq \mathfrak{C}_v$ for all $y \in Y$ and $v\in M_K$.

Fix a point $\gamma \in X(K)$ and a choice of $\omega_\gamma \in K(X)$ defining local coordinates for $X$ near $\gamma$.  An {\bf $M_K$-neighborhood} of $\gamma$ is a collection of open neighborhoods  $U_v$ of $\gamma$ in $X(\bC_v)$, for $v\in M_K$, given locally by $\{|\omega_\gamma|_v < 1\}$ for all but finitely many places $v$.  This definition is independent of the choice of $\omega_\gamma$ uniformizing $X$ near $\gamma$, as a consequence of the following proposition.  

 Let $g$ denote the genus of $X$. For each $\gamma \in X(K)$, choose $\xi^\gamma \in K(X)$ so that $\xi^\gamma$ has a pole of order $2g+1$ at $\gamma$ and no other poles in $X$.  The divisor of a function $h\in K(X)$ is 
 	$$(h) = \sum_{\gamma \in X(\Kbar)} (\ord_\gamma h) \; \gamma.$$

\begin{proposition}  \label{global bound}
For any nonconstant $h \in K(X)$ with $\supp (h) \subset X(K)$, there exists an $M_K$-constant $\mathfrak{c}$ so that 
\begin{multline*}
e^{-\mathfrak{c}_v}  \prod_{\gamma  \in \supp (h)}  \max\{1, |\xi^\gamma(t)|_v\}^{(-\ord_\gamma  h)/(2g+1)}  ~   \\   \leq  \quad   |h(t)|_v   \quad
\leq \quad  e^{\mathfrak{c}_v} \prod_{\gamma \in \supp (h)}   \max\{1, |\xi^\gamma(t)|_v\}^{(-\ord_\gamma  h)/(2g+1)} 
\end{multline*}
for all $t \in X(\Kbar) \setminus \supp (h)$ and $v\in M_K$. Moreover, for each $\gamma \in X(K)$, the notion of $M_K$-neighborhood of $\gamma$ is well defined.
\end{proposition}

\begin{remark}  
The proposition is a theorem of Weil \cite{Weil:arithmetic}, and its proof is contained in \cite[Chapter 10]{Lang:Diophantine} or \cite[Theorem 2.2.11, Remark 2.2.13]{Bombieri:Gubler}, but we include an argument here for completeness.  
\end{remark}

\begin{proof}
For each $\gamma \in \supp (h)$, let $U_\gamma$ be the complement in $X$ of $\supp(h)\setminus\{\gamma\}$ and all zeroes of $\xi^\gamma$, so that $U_\gamma$ is a Zariski-open neighborhood of $\gamma$.  The functions 
	$$h^\gamma  := h^{2g+1} (\xi^\gamma)^{\ord_\gamma h}$$
and $1/h^\gamma$ and $1/\xi^\gamma$ and $\xi^{\gamma'}$ for $\gamma' \not=\gamma$ in $\supp (h)$  are all regular on $U_\gamma$.  Let $U_h = X \setminus \supp(h)$, so that $h$, $1/h$ and each $\xi^\gamma$, $\gamma \in \supp(h)$, are regular on $U_h$.  Note that $\mathcal{U} = \{U_h \} \cup \{U_\gamma: \gamma \in \supp (h)\}$ is an open cover of $X$.  

As in \cite[Chapter 10, Lemma 1.1]{Lang:Diophantine}, there exists a projective embedding of $X$ into $\bP^N$, defined over $K$, so the complement of each coordinate hyperplane in $\bP^N$ intersects $X$ in an open subset of some $U \in \mathcal{U}$.  Indeed, letting $F_U$ be the divisor consisting of the sum of points in the complement of $U \in \mathcal{U}$, we can find effective divisors $H_U$ so that the elements of $\{F_U + H_U :~  U \in \mathcal{U}\}$ are linearly equivalent, and so that there is no point in the intersection of the supports of $F_U + H_U$. (This is because $mH - F_U$ will be very ample for any choice of ample $H$ and every $U \in \mathcal{U}$, for all sufficiently large $m \in \mathbb{N}$.)  The elements $\{F_U + H_U :~ U \in \mathcal{U}\}$ thus induce a morphism $\phi:X\to \mathbb{P}^k$ for some $k$.  Choosing any projective embedding $i:X\hookrightarrow \bP^r$ defined over $K$, for some $r>0$, our desired embedding comes from postcomposing $\phi\times i: X\to \mathbb{P}^k\times \mathbb{P}^r$ with the Segre embedding $\mathbb{P}^k\times \mathbb{P}^r\hookrightarrow \mathbb{P}^{(k-1)(m-1)-1}$.

Let $\mathbb{A}^N(\Kbar)$ denote affine space of dimension $N$, and let 
	$$\|(y_1, \ldots, y_N)\|_v = \max\{|y_1|_v, \ldots, |y_N|_v\}$$
be a $v$-adic norm on $\mathbb{A}^N(\Kbar)$, for each $v \in M_K$.  A collection of subsets $E_v \subset \mathbb{A}^N(\Kbar)$, for $v \in M_K$, is {\bf affine $M_K$-bounded} if each $E_v$ is bounded, and if $E_v$ lies in the unit polydisk $\{y:  \|y\|_v \leq 1\}$ for all but finitely many $v$.  

As in \cite[Chapter 10, Proposition 1.2]{Lang:Diophantine}, we can now cover $X$ by a finite collection $\{{\bf E}_j\}_{j=0}^N$ of affine $M_K$-bounded sets, subordinate to the open cover $\mathcal{U}$.  
Indeed, suppose that $(x_0: x_1: \cdots : x_N)$ are the coordinates of $\bP^N$.  For each $j = 0, \ldots, N$, let $U(j) \in \mathcal{U}$ be an element containing $X \cap \{x_j \not=0\}$.   For each $v \in M_K$, we let $E_{j,v}$ be the set of all points in $\bP^N(\Kbar)$ with projective coordinates $(x_0:  x_1 : \cdots : x_N)$ so that $|x_j|_v$ is maximal.  Then $E_{j,v}$ is the unit polydisk in the affine chart where $x_j\not=0$ with coordinates $y_i = x_i/x_j$.  For each $v\in M_K$, these affine bounded sets cover all of $\bP^N(\Kbar)$ and so also $X$, and the intersection of $E_{j,v}$ with $X$ is a subset of $U(j)$.    
We let ${\bf E}_j$ be the collection $\{E_{j,v}:  v \in M_K\}$, for $j = 0, \ldots, N$.  

Fix $j$.  For $U(j) = U_\gamma$, since $h^\gamma$ and $1/h^\gamma$ are both regular on $U_\gamma$, we have an $M_K$-constant $\mathfrak{g}_\gamma$ such that 
\begin{equation} \label{equal near gamma}
	e^{-\mathfrak{g}_{\gamma,v}} ~ \leq ~ |h^\gamma|_v = | h^{2g+1} (\xi^{\gamma})^{\ord_\gamma h}|_v  ~ \leq ~ e^{\mathfrak{g}_{\gamma,v}}
\end{equation}
on $E_{j,v}$.  It is also the case that $1/\xi^\gamma$ is regular on $U_\gamma$, and so is $\xi^{\gamma'}$ for each $\gamma'\not= \gamma$ in $\supp(h)$, so we can enlarge $\mathfrak{g}_\gamma$ if needed so that 
	$$|\xi^\gamma|_v \geq e^{-\mathfrak{g}_{\gamma,v}}\quad \mbox{ and } \quad |\xi^{\gamma'}|_v \leq e^{\mathfrak{g}_{\gamma,v}} $$
on $E_{j, v}$.  Moreover, we can also arrange that  
	$$|h|_v  \left\{  \begin{array}{ll}   
			\leq  e^{\mathfrak{g}_{\gamma,v}}   & \mbox{if } \ord_\gamma h >0  \\
			\geq e^{-\mathfrak{g}_{\gamma,v} } & \mbox{if } \ord_\gamma h < 0  \end{array}  \right.$$
on $E_{j,v}$, because either $h$ or $1/h$ is regular on $U_\gamma$.  By increasing $\mathfrak{g}_\gamma$ yet again, it follows that 
\begin{eqnarray*}
 e^{-\mathfrak{g}_{\gamma,v}}  \prod_{\beta \in \supp(h)} \max\{1, |\xi^\beta(t)|_v\}^{(-\ord_{\beta}  h)/(2g+1)} &=&
e^{-\mathfrak{g}_{\gamma,v}}   \max\{1, |\xi^\gamma(t)|_v\}^{(-\ord_\gamma  h)/(2g+1)}   \\
&\leq&   |h(t)|_v    \\
&\leq&  e^{\mathfrak{g}_{\gamma,v}}   \max\{1, |\xi^\gamma(t)|_v\}^{(-\ord_\gamma  h)/(2g+1)}  \\
&=&  e^{\mathfrak{g}_{\gamma,v}}  \prod_{\beta \in \supp(h)} \max\{1, |\xi^\beta(t)|_v\}^{(-\ord_{\beta}  h)/(2g+1)} 
\end{eqnarray*}
for all $t \in E_{j,v}$ where $U(j) = U_\gamma$.  

Similarly for $U(j) = U_h$, we can find an $M_K$-constant $\mathfrak{s}$ so that  
	$$e^{-\mathfrak{s}_{v}} \leq |h|_v \leq e^{\mathfrak{s}_{v}}$$
and 
	$$|\xi^\gamma|_v \leq e^{\mathfrak{s}_v}$$
on $E_{j, v}$, for all $\gamma \in \supp(h)$.  This completes the proof of the first statement of the proposition.

To see that the notion of $M_K$-neighborhood is well defined, we fix $\gamma_0 \in X(K)$ and choose any $\omega_0 \in K(X)$ with a simple zero at $\gamma_0$.   For the covering $\mathcal{U}$ of $X$ associated to $\omega_0$, note that $U_{\gamma_0}$ is the unique element containing $\gamma_0$.  So for each $v$ and $j$, if the set $E_{j,v}$ contains $\gamma_0$, then it must lie in $U_{\gamma_0}$.  The inequality \eqref{equal near gamma} implies that $|\omega_0|_v^{2g+1} = |\xi^{\gamma_0}|_v^{-1}$ on such $E_{j,v}$, for all but finitely many $v$.  On the other hand, we also have that if $|\xi^{\gamma_0}(t)|_v > 1$ at a point $t \in E_{j,v}$, for some $j$, then $E_{j,v}$ is contained in $U_{\gamma_0}$ for all but finitely many $v$ (because $|\xi^{\gamma_0}|_v \leq 1$ on the $E_{j,v}$'s in the other elements of $\mathcal{U}$).  In other words, any $M_K$-neighborhood of $\gamma_0$ defined by $\omega_0$ coincides with $\{t\in X(\bC_v): |\xi^{\gamma_0}(t)|_v > 1\}$ for all but finitely many places $v$ of $K$.  This completes the proof of the proposition.  
\end{proof}  

\bigskip
\section{Escape rates and Weil heights}

Throughout this section, we fix $f: \bP^1\to \bP^1$ of degree $d \geq 2$, defined over $k = K(X)$, and any point $a \in \bP^1(k)$.  

\subsection{The singular set $\cS(F,A)$ in $X$}  \label{lifts}
Working in homogeneous coordinates on $\bP^1$, we let
	$$F = (P,Q):  \A^2 \to \A^2$$
be a {\bf homogeneous lift} of $f$, with homogeneous polynomials $P(z,w) \in k[z,w]$ and $Q(z,w) \in k[z,w]$ of degree $d$ having no common zeroes in $\bP^1(\kbar)$, so that $f(z) = P(z,1)/Q(z,1)$ in local coordinates.  Choose a {\bf lift} $A = (\alpha, \beta) \in k^2\setminus \{(0,0)\}$ of $a = (\alpha:\beta) \in \bP^1(k)$.  For each $\gamma \in X(\Kbar)$, we let 
	$$\ord_\gamma F = \min\{\ord_\gamma c:  \mbox{ coefficients } c \mbox{ of } P \mbox{ and } Q\}$$
and
	$$\ord_\gamma A = \min\{ \ord_\gamma \alpha, \ord_\gamma \beta\}.$$
We let $\Res F \in k^*$ denote the homogeneous resultant of $P$ and $Q$; see, for example, \cite[\S2.4]{Silverman:dynamics}.  

Set
\begin{equation} \label{sing F}
	\cS(F) = \{\gamma \in X(\Kbar):  \ord_\gamma F \not=0 \mbox{ or } \ord_\gamma \Res F \not= 0\}
\end{equation}
and
\begin{equation} \label{sing F A}
	\cS(F,A) = \cS(F) \cup \{\gamma \in X(\Kbar):  \ord_\gamma A \not= 0\}
\end{equation}
Note that $\cS(F,A)$ is a finite set.  

\begin{convention}  \label{K choice}
We enlarge the number field $K$, if needed, so that $\cS(F,A) \subset X(K)$. 
\end{convention}

\subsection{Geometric escape rates and a divisor on $X$}

Recall here that throughout we identify the places of $k$ with the points $\gamma\in X(\Kbar)$, with a slight abuse of terminology. 
For each $\gamma \in X(\Kbar)$, we work with the absolute value on $k$ defined by 
	$$|z|_\gamma := e^{-\ord_\gamma z},$$
and the norm $\| \cdot \|_\gamma$ on $k^2$ given by 
	$$\|(z,w)\|_\gamma = \max\{|z|_\gamma, |w|_\gamma\}.$$
There is a constant $C_\gamma \geq 1$ so that 
\begin{equation}  \label{gamma bounds}
	C_\gamma^{-1} \|(z,w)\|_\gamma^d \leq \|F(z,w)\|_\gamma \leq C_\gamma \|(z,w)\|_\gamma^d
\end{equation}
for all $(z,w) \in k^2$; we can take $C_\gamma = 1$ for all $\gamma \not\in \cS(F)$ \cite[Proposition 5.57]{Silverman:dynamics}.

The {\bf escape rate} of $A$ for $F$ at $\gamma$ is the quantity
\begin{equation} \label{geometric escape}
	G_{F,\gamma}(A) = \lim_{n\to\infty}  \frac{1}{d^n} \log\|F^n(A)\|_\gamma.
\end{equation} 
It exists in $\R$, by \eqref{gamma bounds}, and it is equal to 0 for all $\gamma \not\in \cS(F,A)$; see, e.g., \cite[Proposition 5.58]{Silverman:dynamics}.  We define an $\R$-divisor by  
\begin{equation} \label{D(F,A)}
	D(F,A) = \sum_{\gamma \in X(\Kbar)}  G_{F,\gamma}(A)  \;  \gamma.
\end{equation}
The support of $D(F,A)$ is contained in $\cS(F,A)$ and so in $X(K)$ by Convention \ref{K choice}.
If we had chosen different lifts of $f$ and $a$, say $cF$ and $bA$ for $c, b \in k^*$, then 
\begin{equation} \label{geometric change}
	G_{cF, \gamma}(bA) = G_{F,\gamma}(A) - \frac{1}{d-1} \ord_\gamma c - \ord_\gamma b.
\end{equation}
It follows that $D(cF, bA)$ and $D(F,A)$ are linearly equivalent $\R$-divisors on $X$.

\subsection{A Weil height associated to $D(F,A)$}  \label{Weil for (F,A)}
Let $g$ denote the genus of $X$.   For each $\gamma \in X(K)$, choose a meromorphic function $\xi^{\gamma} \in K(X)$ so that $\xi^{\gamma}$ has a pole of order $2g+1$ at $\gamma$ and no other poles.  

Let $D = D(F,A)$ be defined by \eqref{D(F,A)}, and recall that $\supp D(F,A) \subset X(K)$ by Convention \ref{K choice}.  For each place $v$ of $K$, we define a function on $X(\bC_v) \setminus \supp D$ by 
	$$\lambda_{D,v}(t) =  \sum_{\gamma \in \cS(F,A)} G_{F,\gamma}(A)  \, \frac{\log^+|\xi^\gamma(t)|_v}{2g+1}.$$
This function extends continuously to the Berkovich analytificiation $X_v^{an} \setminus \supp D$.

A Weil height for $D$ can be defined by
\begin{equation} \label{Weil height}
	h_D(t) = \frac{1}{[K:\bQ]}  \frac{1}{|\Gal(\Kbar/K)\cdot t|} \; \sum_{x\in \Gal(\Kbar/K)\cdot t} \; \sum_{v \in M_K} N_v \, \lambda_{D,v}(x)
\end{equation}
for all $t \in X(\Kbar)\setminus\supp D$, and we may set $h_D(t)=0$ for $t\in\supp D$.  
This $h_D$ is indeed a Weil height associated to the $\R$-divisor $D$, as it is an $\R$-linear combination of Weil heights built from the local functions
	$$\frac{1}{2g+1} \log^+|\xi^{\gamma}(t)|_v$$
at each place $v$ of $K$, associated to the divisor $\gamma$.

\subsection{Arithmetic escape rates} \label{aer}
For each place $v$ of the number field $K$, we define a norm $\| \cdot \|_v$ on $\Kbar^2$ by 
	$$\|(z,w)\|_v = \max\{|z|_v, |w|_v\}.$$
For each $t \in X(\Kbar) \setminus \cS(F)$, we let $F_t$ denote the specializations of $F$.  We continue to use the collection of functions $\{\xi^{\gamma}:  \gamma \in \cS(F)\}$ from \S\ref{Weil for (F,A)}.  The following proposition appears in various forms in the literature, e.g., in \cite[Proposition 5.57]{Silverman:dynamics} \cite[Lemma 10.1]{BRbook} \cite[Lemma 3.3]{D:stableheight}, but we require an adelic version for our theorems.  Recall that $\cS(F) \subset X(K)$ by Convention \ref{K choice}.

\begin{proposition}  \label{F v bounds}
For each $\gamma \in \cS(F)$, choose $\beta^{\gamma} \in k$ so that $\ord_\gamma (\beta^\gamma  F) = 0$.  There is an $M_K$-constant $\mathfrak{b}$ and an $M_K$-neighborhood $\mathfrak{U}^\gamma$ of $\gamma$ in $X$ so that
$$e^{-\mathfrak{b}_v}  |\xi^{\gamma}(t)|_v^{- (\ord_\gamma \Res (\beta^\gamma F))/(2g+1)} ~ \leq ~ \frac{|\beta^\gamma_t|_v \| F_t(z,w)\|_v}{\|(z,w)\|_v^d} ~\leq~ e^{\mathfrak{b}_v}   $$
for all $v\in M_K$, $t \in \mathfrak{U}_v^\gamma \setminus \{\gamma\}$, and for all $(z,w) \in (\bC_v)^2\setminus\{(0,0)\}$. 
We can choose the $M_K$-neighborhoods $\mathfrak{U}^\gamma$ for $\gamma \in \cS(F)$ to be pairwise disjoint.  Moreover, given any $M_K$-neighborhoods $\mathfrak{U}^\gamma$, for $\gamma \in \cS(F)$, there exists an $M_K$-constant $\mathfrak{c}$ so that 
$$e^{-\mathfrak{c}_v}   ~ \leq  ~ \frac{\|F_t(z,w)\|_v}{\|(z,w)\|_v^d} ~ \leq ~  e^{\mathfrak{c}_v}$$
for all $v\in M_K$, $t \in X(\bC_v) \setminus \left( \bigcup_{\gamma \in \cS(F)} \mathfrak{U}^\gamma_v \right)$, and for all $(z,w) \in (\bC_v)^2\setminus\{(0,0)\}$.
\end{proposition}

\begin{proof}
Recall that $F = (P,Q)$ for homogeneous polynomials $P$ and $Q$ of degree $d$ with coefficients in $k$.  
By our choice of $\beta^\gamma$, there is an $M_K$-neighborhood $\mathfrak{U}^\gamma$ and an $M_K$-constant $\mathfrak{b}$ so that 
\begin{align}\label{bounded coefficients}
e^{-\mathfrak{b}_v} ~\leq ~\max\{ |c_t|_v:  \mbox{ coefficients } c \mbox{ of } \beta^\gamma P \mbox{ and } \beta^\gamma Q\}  ~ \leq~ e^{\mathfrak{b}_v} 
\end{align}
for each $t \in \mathfrak{U}_v^\gamma$ and each $v \in M_K$, by Proposition \ref{global bound}. 
By increasing the constant $\mathfrak{b}$, the upper bound on $\|\beta^\gamma_t F_t(z,w)\|_v/\|(z,w)\|_v^d$ follows from the triangle inequality. 

We can enlarge $\mathfrak{b}$ at the archimedean places, if needed, so that 
	$$\frac{\|\beta^\gamma_t F_t(z,w)\|_v}{\|(z,w)\|_v^d}  ~ \geq  ~  e^{-\mathfrak{b}_v} |\Res(\beta^\gamma_t F_t)|_v$$
for each $t \in \mathfrak{U}_v^\gamma$, for all $v \in M_K$, applying \cite[Proposition 5.57]{Silverman:dynamics} and \cite[Lemma 3.3]{D:stableheight}.  Applying Proposition \ref{global bound} again, this time to $\Res(\beta^\gamma F)$, shrinking the $M_K$-neighborhood and increasing the $M_K$-constant $\mathfrak{b}$ again if necessary, we have
	$$\frac{\|\beta^\gamma_t F_t(z,w)\|_v}{\|(z,w)\|_v^d}  ~ \geq  ~  e^{-\mathfrak{b}_v}  |\xi^{\gamma}(t)|_v^{- (\ord_\gamma \Res (\beta^\gamma F))/(2g+1)}$$
for all $t \in \mathfrak{U}_v^\gamma$ and each $v \in M_K$.

The final statements of the proposition follow from the same combination of Proposition \ref{global bound} with \cite[Proposition 5.57]{Silverman:dynamics}, because the coefficients of $F$ will have no poles and $\Res F$ will have no poles or zeroes outside of $\cS(F)$.  
\end{proof}

Similar to the geometric escape rates of \eqref{geometric escape}, we can define arithmetic escape rates, working at each place $v$ of the number field $K$.  For each $v \in M_K$, the {\bf escape rate function for the pair $(F,A)$} is defined on $X(\Kbar)\setminus \cS(F,A)$ by
\begin{equation} \label{arithmetic escape}
	G_{F_t,v}(A_t) = \lim_{n\to\infty}  \frac{1}{d^n} \log\|F_t^n(A_t)\|_v.
\end{equation}
It exists in $\R$ for all $t \in X(\Kbar)\setminus \cS(F,A)$ by Proposition \ref{F v bounds}; see, e.g., \cite[Proposition 5.58]{Silverman:dynamics}.  The proof of convergence for \eqref{arithmetic escape} shows it is locally uniform in $t$, so that $G_{F_t,v}(A_t)$ extends to a continuous function $t \in X(\bC_v)\setminus \cS(F,A)$ at each place $v \in M_K$.  In fact, it extends to be continuous on the Berkovich analytification $X_v^{an}\setminus \cS(F,A)$; see, e.g., \cite[pp.~295--296]{BRbook} where the escape rate is ``Berkovich-ized".  If we had chosen different lifts of $f$ and $a$, say $cF$ and $bA$ for $c, b \in k^*$, then 
\begin{equation} \label{arithmetic change}
	G_{c_tF_t,v}(b_tA_t) = G_{F_t,v}(A_t) + \frac{1}{d-1} \log|c_t|_v + \log|b_t|_v
\end{equation}
for $t \in X(\Kbar) \setminus (\cS(F,A) \cup \cS(cF, bA))$.  

These escape rate functions provide local height expressions for the canonical height $\hat{h}_{f_t}$ evaluated at $a_t$.  In particular, we have 
	$$\hat{h}_{f_t}(a_t) = \frac{1}{[K:\bQ]}  \frac{1}{|\Gal(\Kbar/K)\cdot t|} \sum_{x\in \Gal(\Kbar/K)\cdot t} \sum_{v\in M_K} N_v \,G_{F_x,v}(A_x) $$
for all $t \in X(\Kbar) \setminus \cS(F,A)$.  See, for example, \cite[Theorem 5.59]{Silverman:dynamics}.  Note that the sum over all places of $K$ is independent of the choice of lifts $F$ and $A$, by the product formula.  
	
\subsection{Variation of canonical height}  \label{proof goals}
Recall that we are trying to understand if 
	$$\hat{h}_{f_t}(a_t) - h_D(t)$$
is bounded, as claimed in Theorem \ref{dyn VCH}, where $h_D$ is a choice of Weil height for $D = D(F,A)$ defined by \eqref{D(F,A)}.  Recalling that any two choices of Weil height for the same divisor are bounded from one another (and in fact, $M_K$-bounded) it suffices to work with the Weil height constructed in \eqref{Weil height}.  Assuming that the point $a \in \bP^1(k)$ is totally Fatou for $f$, a hypothesis which will be defined and examined in the next Section, we aim to prove three things:  
\begin{enumerate}
\item  that the local geometric height $G_{F,\gamma}(A)$ is in $\bQ$ at all points $\gamma \in X(\Kbar)$, so that the divisor $D = D(F,A)$ of \eqref{D(F,A)} will be a $\bQ$-divisor; 
\item that for all places $v$ of $K$, the $v$-adic functions 
	$$V_v(t) :=  G_{F_t,v}(A_t) -  \sum_{\gamma \in \cS(F,A)}  G_{F,\gamma}(A) \, \frac{\log^+|\xi^\gamma(t)|_v}{2g+1}$$
on $X(\bC_v) \setminus \cS(F,A)$ extend to bounded -- and in fact continuous -- functions on the Berkovich analytification $X^{an}_{\bC_v}$; and
\item that the sum 
	$$\hat{h}_{f_t}(a_t) - h_D(t) ~ = ~ \frac{1}{|\Gal(\Kbar/K)\cdot t|} \sum_{x\in \Gal(\Kbar/K)\cdot t}  \sum_v N_v \, V_v(x)$$ 
is uniformly bounded over all points $t \in X(\Kbar) \setminus \cS(F,A)$.   
\end{enumerate}

\bigskip
\section{The non-archimedean Fatou set}
\label{Fatou section}

Throughout this section, we fix $f: \bP^1\to \bP^1$ of degree $d \geq 2$, defined over $k = K(X)$.  For each fixed $\gamma \in X(\Kbar)$, we let $k_\gamma$ be the completion of $k$ with respect to the valuation $\ord_\gamma$, and let $\mathbb{L}_\gamma$ be the completion of an algebraic closure of $k_\gamma$.   In this section, we introduce and study the totally Fatou condition that is assumed for Theorem \ref{dyn VCH}, and we prove Theorems \ref{dense} and \ref{modification}.

\subsection{The Fatou set} \label{Fatou def}
Fix $\gamma\in X(\Kbar)$.  Let $d_\gamma(x,y)$ denote the chordal distance between $x$ and $y$ in $\bP^1(\bL_\gamma)$.  Explicity, if $x = (x_1:x_2)$ and $y = (y_1:y_2)$, then 
	$$d_\gamma(x,y) = \frac{|x_1y_2 - x_2y_1|_\gamma}{\max\{|x_1|_\gamma, |x_2|_\gamma\} \max\{|y_1|_\gamma, |y_2|_\gamma\}}.$$
The {\bf non-archimedean Fatou set of $f$ at $\gamma$} is the set $\Omega_\gamma(f)$ of all points $x \in \bP^1(\bL_\gamma)$ for which we can find an open disk $D_x$ containing $x$ so that the family of functions $\{f^n|{D_x}\}$ is equicontinuous in the distance $d_\gamma$.  See, for example, \cite[Chapter 5]{Benedetto:book}. Its complement $\bP^1(\bL_{\gamma})\setminus\Omega_{\gamma}(f)$ is the {\bf non-archimedean Julia set of $f$ at $\gamma$}.

This Fatou set $\Omega_\gamma(f)$ will be all of $\bP^1(\bL_\gamma)$ at $\gamma$ where $f$ has good reduction.  In our case, this implies that $\Omega_\gamma(f) = \bP^1(\bL_\gamma)$ for all $\gamma \not\in \cS(F)$, the singular set defined in \eqref{sing F}, for any choice of homogeneous polynomial lift $F$ of $f$. 

A point $a \in \bP^1(k)$ is {\bf totally Fatou for $f$} if $a \in \Omega_\gamma(f)$ at all $\gamma \in X(\Kbar)$.

\subsection{Hole-avoiding pairs} \label{hole-avoiding}
Now fix a point $a \in \bP^1(k)$.  Fix $\gamma\in X(\Kbar)$, and choose homogeneous polynomial lift
	$$F = (P,Q):  \A^2 \to \A^2$$
of $f$ over $k$ and lift 
	$$A = (\alpha, \beta) \in k^2$$
of $a$ so that 
\begin{equation} \label{normalized lifts}
	\ord_\gamma F = \ord_\gamma A = 0,
\end{equation}
where $\ord_\gamma F$ and $\ord_\gamma A$ are defined in \S\ref{lifts}, so that the specializations $F_\gamma$ and $A_\gamma$ are well defined.
The {\bf holes} of $f$ at $\gamma$ are the points $x = (x_1:x_2) \in \bP^1(\Kbar)$ for which $F_\gamma(x_1,x_2) = (0,0)$.  Holes exist if and only if $\Res F_\gamma = 0$. We say that the pair $(f,a)$ is {\bf hole-avoiding at $\gamma$} if the specializations satisfy
	$$F_\gamma^n(A_\gamma) \not= (0,0)$$
for all $n \geq 0$.  In particular, the pair $(f,a)$ is hole-avoiding at $\gamma$ for all points $a \in \bP^1(k)$ if $\Res F_\gamma \not=0$. It is easy to check that this definition is independent of the choice of lifts $F$ and $A$, as long as they satisfy \eqref{normalized lifts}.

\begin{example}  \label{translation example}
Consider 
	$$f(z) = \frac{z(z-1)}{z-t}$$ 
over $k = \bQ(t)$, at $t=0$.  The polynomial map $F(z,w) = (z(z-w), (z-tw)w)$ specializes to $F_0(z,w) = (z(z-w), zw)$.  The point $0 = (0:1) \in \bP^1(\Qbar)$ is the unique hole for $f$.  A point $a \in \bP^1(k)$ will therefore {\em fail to be} hole-avoiding at $t=0$ if and only if it specializes to $a_0 \in \bZ_{\geq 0}$.  Indeed, for $a_0 = n_0 \in \bZ_{\geq 0}$, the iterates of the lift $A = (a,1)$ will satisfy $F_0^{n_0+1}(A_0) = (0,0)$.  
\end{example}

\begin{example} \label{I(d) example}
Consider
	$$f(z)= z^2 + 1/t$$
over $k = \bQ(t)$, at $t=0$.  The polynomial map $F(z,w) = (t z^2 + w^2, t w^2)$ specializes to $F_0(z,w) = (w^2, 0)$.  The point $\infty = (1:0)  \in \bP^1(\Qbar)$ is the unique hole for $f$.  There are no hole-avoiding points in $\bP^1(k)$, because $F_0^2(z,w) = F_0(w^2,0) = (0,0)$ for all $(z,w) \in \Kbar^2$. 
\end{example}

We may view $f: \bP^1 \to \bP^1$ over $k$ as a rational map 
	$$\tilde{f}: X \times \bP^1 \dashrightarrow X \times \bP^1$$
of the surface $X \times \bP^1$ to itself, defined over the number field $K$ by $(t, z) \mapsto (t, f_t(z))$.  We may view the point $a \in \bP^1(k)$ as a section of the projection $X\times\bP^1 \to X$, also defined over $K$.  The following is immediate from the definitions:  

\begin{lemma}  \label{holes are indeterminate}  
A pair $(f,a)$ is hole-avoiding at $\gamma \in X(\Kbar)$ if and only if the iterates $f^n(a)$ (as sections of the fibered surface $X \times \bP^1 \to X$) are disjoint from the indeterminacy locus of the induced map $\tilde{f}:  X \times \bP^1 \dashrightarrow X \times \bP^1$ within the fiber $\{\gamma\}\times \bP^1$, for all $n\geq 0$.
\end{lemma}

Note that all of the indeterminacy points of $\tilde{f}$ in $X \times \bP^1$ are contained in the fibers over $\cS(F) \subset X$ for every choice of homogeneous polynomial lift $F$.  The term ``hole" for an indeterminacy point was first used in \cite{D:measures};  it was meant to capture the idea that the mass of the measures of maximal entropy in the family $f_t$, for $t \in X(\bC)\setminus \cS(F)$, was ``falling into the holes" of $f$ and its iterates $f^n$ at $t = \gamma$. The same condition appears in \cite{Mavraki:Ye}.  

\subsection{The action of $f$ on the Berkovich projective line} \label{Gamma stuff}
Fix $\gamma\in X(\Kbar)$. 
We now reinterpret the notion of hole-avoiding in the language of the Berkovich projective line defined over the field $\bL_\gamma$, which we denote by $\bP^{1,an}_\gamma$, and the extension of $f$ to a dynamical system on $\bP^{1,an}_\gamma$.   A good basic reference for the dynamics of $f$ on $\bP^{1,an}_\gamma$ is \cite{Benedetto:book}.  Note that the definition of hole-avoiding extends naturally to elements $a \in \bP^1(k_\gamma)$, for $k_\gamma$ the completion of $k$ at $\gamma$, with a lift $A$ chosen in $(k_\gamma)^2\setminus \{(0,0)\}$.  

By definition, the {\bf hole-directions} for $f$ from a Type II point $\zeta \in \bP^{1,an}_\gamma$ are the connected components of $\bP^{1,an}_\gamma \setminus \{\zeta\}$ that intersect the set of preimages $f^{-1}(\zeta)$.  When $\zeta = \zeta_G$ is the Gauss point in $\bP^{1,an}_\gamma$,  the hole-directions correspond to the holes of $f$ at $\gamma$, via the natural identification of the connected components of $\bP^{1,an}_\gamma \setminus \{\zeta\}$ with $\bP^1(\Kbar)$.  See, for example, \cite[\S7.5]{Benedetto:book}. (In particular, if $f(\zeta)=\zeta$, the hole-directions coincide with the ``bad" directions of \cite[Theorem 7.34]{Benedetto:book}.)  This implies:

\begin{lemma}  \label{holes are bad}  
Let $a$ be an element of $\bP^1(k_\gamma)$.  The pair $(f,a)$ is hole-avoiding at $\gamma$ if and only if neither $a$ nor any iterate $f^n(a)$ lies in a hole-direction for $f$ from the Gauss point $\zeta_G$ in $\bP^{1,an}_\gamma$.   
\end{lemma}

Let $\Omega^{an}_\gamma(f)$ be the Berkovich Fatou set of $f$ in $\bP^{1,an}_\gamma$; see, e.g., \cite[Chapter 8]{Benedetto:book}.  Any open set $U \subset \bP^{1,an}_\gamma$ that intersects the Berkovich Julia set $J^{an}_\gamma(f) := \bP^{1,an}_\gamma \setminus \Omega^{an}_\gamma(f)$ has the property that the union $\mathcal{U} = \bigcup_{n\geq 0} f^n(U)$ is dense in $\bP^{1,an}_\gamma$; in fact, the set $\mathcal{U}$ omits at most 2 points, both in $\bP^1(\mathbb{L}_\gamma)$ \cite[Theorem 8.15]{Benedetto:book}.  Recall that $\Omega^{an}_\gamma(f) \cap \bP^1(\bL_\gamma) = \Omega_\gamma(f)$, the non-archimedean Fatou set as we have defined it in \S\ref{Fatou def}.  

\begin{lemma} \label{Julia holes}
For any Type II point $\zeta \in \bP^{1,an}_\gamma$, the points of the non-archimedean Julia set of $f$ in $\bP^1(\mathbb{L}_\gamma)$ are contained in the union of the hole-directions from $\zeta$ for $f^n$, over all $n\geq 1$.   
\end{lemma}

\begin{proof} 
Let $U$ be a connected component of $\bP^{1,an}_\gamma \setminus \{\zeta\}$.  If $U$ has non-empty intersection with the non-archimedean Julia set $J^{an}_\gamma(f) \cap \bP^1(\bL_\gamma)$, then the iterates of $U$ must contain all Type II points, including $\zeta$ itself  \cite[Theorem 8.15]{Benedetto:book}.  So $U$ must be a hole-direction from $\zeta$ for some iterate of $f$.  
\end{proof}

We are now ready to prove the following result, needed to analyze the dynamics of totally Fatou points for the proof of Theorem \ref{dyn VCH} and Theorem \ref{local dyn VCH}:

\begin{theorem} \label{B point}
Fix any $f: \bP^1\to \bP^1$ of degree $d \geq 2$, defined over $k = K(X)$, a point $\gamma \in X(K)$, and any point $a \in \bP^1(k_\gamma)$.  The point $a$ lies in the non-archimedean Fatou set $\Omega_\gamma(f)$ if and only if there exist a change of coordinates $B \in \mathrm{PGL}_2(k)$ and iterates $f^n$ and $f^m$ so that the pair $(B f^n B^{-1}, B(f^m(a)))$ is hole-avoiding at $\gamma$.  
\end{theorem}

We shall see that one implication is straightforward from the definitions, that the existence of the hole-avoiding pair implies that $a \in \Omega_\gamma(f)$.  To prove the converse implication, assuming $a \in \Omega_\gamma(f)$, we follow the proof of \cite[Theorem D]{DF:degenerations2}, which itself uses the Rivera-Letelier classification of Berkovich Fatou components in the Berkovich space $\bP^{1,an}_\gamma$ \cite{RiveraLetelier:Asterisque} \cite[Appendix]{DF:degenerations2} and the Benedetto wandering domains theorem \cite{Benedetto:wandering}, while also keeping track of the orbit of the point $a$.

In the language of \cite{DF:degenerations2}, given a finite set $\Gamma$ of Type II points in $\bP^{1,an}_\gamma$, a connected component $U$ of $\bP^{1,an}_\gamma\setminus \Gamma$ is called a {\bf $J$-component} for $\Gamma$ if $f^n(U)\cap\Gamma \not= \emptyset$ for some $n>0$.  The Julia set $J^{an}_\gamma(f)$ is contained in the union of the $J$-components and $\Gamma$ \cite[Proposition 2.5]{DF:degenerations2}.   The other connected components of $\bP^{1,an}_\gamma\setminus \Gamma$ are called {\bf $F$-components}.  An $F$-component $U$ for $\Gamma$ is called an {\bf $F$-disk} if it is a Berkovich disk; it is {\bf wandering} if the iterates $f^n(U)$ lie in pairwise disjoint $F$-components for all $n$.

A pair $(f, \Gamma)$ is {\bf analytically stable} if, for each $\zeta \in \Gamma$, we have either $f(\zeta) \in \Gamma$ or $f(\zeta)$ is contained in an $F$-component for $\Gamma$.

The {\bf $k$-split} Type II points $\zeta$ are those in the $\mathrm{PGL}_2(k)$-orbit of the Gauss point $\zeta_G$.  These are the Type II points that have $k$-rational points in infinitely many connected components of $\bP^{1,an}\setminus \{\zeta\}$.  Our proof strategy for Thoerem \ref{B point} also gives the following statement, which will be used in our proof of Theorem \ref{modification}. 

\begin{theorem}  \label{stable modification}
Fix any $f: \bP^1\to \bP^1$ of degree $d \geq 2$ defined over $k = K(X)$, a point $\gamma \in X(K)$, and any Fatou point $a \in \Omega_\gamma(f)\cap \bP^1(k_\gamma)$.   For any finite set of Type II points $\Gamma$, there exists a finite set $\Gamma' \supset \Gamma$ so that the pair $(f,\Gamma')$ is analytically stable and each point $f^n(a)$ of the orbit of $a$ lies in an $F$-disk for $\Gamma'$.  Moreover, if the elements of $\Gamma$ are $k$-split, then we can choose $\Gamma'$ so that its elements are also $k$-split.  
\end{theorem}

\begin{remark}
In \cite[Theorem D]{DF:degenerations2}, the existence of an analytically stable $\Gamma' \supset \Gamma$ for the map $f$ is proved, but without the additional conclusion about the orbit of the Fatou point $a$.
\end{remark}

\begin{proof}[Proof of Theorems \ref{B point} and \ref{stable modification}]
Fix $a \in \bP^1(k_\gamma)$, and assume first that there exist $B \in \mathrm{PGL}_2(k)$ and integers $n \geq 1$ and $m\geq 0$ so that the pair $(B f^n B^{-1}, B(f^m(a)))$ is hole-avoiding at $\gamma$.  Let $\zeta_B = B^{-1}(\zeta_G) \in \bP^{1,an}_\gamma$, where $\zeta_G$ is the Gauss point.  Then by Lemma \ref{holes are bad}, 
the point $f^m(a)$ and all iterates $f^{jn + m}(a)$, for $j\geq 0$, do not lie in the hole-directions of $f^n$ from $\zeta_B$.  But the existence of such an orbit implies that either $f^n(\zeta_B) = \zeta_B$ or that $f^n(\zeta_B)$ lies in a direction from $\zeta_B$ which is not a hole-direction (for otherwise all points would either be in a hole-direction or mapped into a hole-direction under one iterate).   If $f^n(\zeta_B) = \zeta_B$, then the hole-directions from $\zeta_B$ for an iterate $f^{jn}$, with $j\geq 1$, coincide with directions that are mapped to the hole-directions for $f^n$ by some $f^{\ell n}$ with $\ell < j$; if $f^n(\zeta_B) \not= \zeta_B$, then the hole-directions for the iterates $f^{jn}$ must coincide with the holes for $f^n$, for all $j\geq 1$.  In either case, it then follows from Lemma \ref{Julia holes} that $f^m(a)$ is not in $J^{an}_\gamma(f) \cap \bP^1(\bL_\gamma)$.  In other words, $a$ must be an element of the Fatou set $\Omega_\gamma(f)$.  This proves one implication of Theorem \ref{B point}.

To prove the converse implication in Theorem \ref{B point}, we need to find a coordinate change $B$ with good properties.  We do this by constructing a Type II point $\zeta_B$ with the desired properties, and then we will choose any $B \in \mathrm{PGL}_2(k)$ sending $\zeta_B$ to the Gauss point $\zeta_G$. Along the way, we will prove Theorem \ref{stable modification}.

Let $\Gamma$ be any finite set of Type II points.  From \cite[Theorem D]{DF:degenerations2}, we know that there is a finite set of Type II points $\Gamma'\supset \Gamma$ (which can be chosen to be $k$-split if $\Gamma$ is $k$-split) so that the pair $(f, \Gamma')$ is analytically stable.  More precisely, the main theorem of \cite[\S3.2]{DF:degenerations2} states that, for every $\zeta \in \Gamma'$, one of the following three cases must hold:
\begin{enumerate}
\item the orbit of $\zeta$ lies in $\Gamma'$, and $f^k(\zeta) = f^\ell(\zeta)$ for some $\ell > k \geq 0$; 
\item some iterate of $\zeta$ lies in a wandering $F$-disk for $\Gamma'$, with a periodic boundary point $\zeta' \in \Gamma'$; or
\item some iterate of $\zeta$ lies in an $F$-component for $\Gamma'$ that contains an attracting periodic point.
\end{enumerate}  

Now fix a point $a \in \Omega_\gamma(f)$.  Choose a Berkovich disk $D_a$ containing $a$ and contained in $\Omega^{an}_\gamma(f)$, with $k$-split Type II boundary point $\zeta_a$.  Choose $D_a$ small enough so that the elements of $\Gamma'$ are disjoint from the forward iterates $f^j(D_a)$ for all $j\geq 0$; this is possible because $D_a \subset \Omega_{\gamma}^{an}(f)$.  Following the proof in \cite[\S3.2]{DF:degenerations2}, we use the classification of Fatou components (see \cite[Theorem A.1]{DF:degenerations2}) to analyze the orbit of the disk $D_a$, to choose a distinguished $k$-split Type II point to be $\zeta_B$, and to increase the set $\Gamma'$ further so it contains $\zeta_a$ and $\zeta_B$ and remains analytically stable.   

First assume that $a$ (and so also $D_a$) lies in a wandering Fatou component $U$.  From \cite[Theorem 5.1]{Benedetto:wandering}, there exists $m\geq 0$ so that $V = f^m(U)$ is a wandering Berkovich disk with periodic and $k$-split boundary point; see also \cite[Proposition 3.8]{DF:degenerations2}.  Enlarging $m$ if necessary, we can assume that $V$ does not contain $\zeta_a$ and that the union $\bigcup_{j\geq 0} f^j(V)$ is disjoint from $\Gamma'$.  Let $\zeta_B$ be the boundary point of $V$ and $n \geq 1$ the period of $\zeta_B$.  We then increase $\Gamma'$ to include the iterates $\zeta_a, f(\zeta_a), \ldots, f^{m-1}(\zeta_a)$ and the periodic orbit of $\zeta_B$.  Then the point $f^j(a)$ lies in an $F$-disk for $\Gamma'$ for all $j\geq 0$, the point $f^m(a)$ lies in a wandering $F$-disk for $\Gamma'$ with boundary point $\zeta_B$, and $(f,\Gamma')$ is analytically stable.  

Now suppose that $a$ lies in the basin of attraction of an attracting periodic point $p$ of period $n$.  Then there exists an integer $s \geq 0$, so that $f^s(D_a)$ is contained in the periodic Fatou component containing $p$.  Note that $p$ must be in the completion $\bP^1(k_\gamma)$, because the iterates $f^{jn + s}(a) \in \bP^1(k_\gamma)$ converge to $p$ as $j \to \infty$.  Thus, there exists a small disk $D_p$ around $p$ with $k$-split boundary point $\zeta_p$ that does not contain $f^s(\zeta_a)$ nor any element of $\Gamma'$ and so that $f^n(D_p) \subsetneq D_p$.   Now choose $m > s$ so that $f^m(a) \in D_p$ while $f^{m-1}(\zeta_a) \not\in D_p$.  We include the orbit $\zeta_a, f(\zeta_a), \ldots, f^{m-1}(\zeta_a)$ in $\Gamma'$, and we also include $\zeta_p, f(\zeta_p), \ldots, f^{n-1}(\zeta_p)$.  Then we set $\zeta_B = \zeta_p$.  Again it follows that the point $f^j(a)$ lies in an $F$-disk for $\Gamma'$ for all $j\geq 0$, the point $f^m(a)$ in an $F$-disk for $\Gamma'$ containing the attracting periodic point, and $(f,\Gamma')$ is analytically stable. 

The final case is where an iterate $f^m(D_a)$ lies in a periodic Rivera domain $V$.  If $\zeta_a$ is preperiodic, we include its forward orbit in $\Gamma'$.  Let $P$ be the subset of the closure $\overline{V}$ which is periodic; as explained in \cite[Lemma 3.10]{DF:degenerations2}, the set $P$ is a closed and connected subset of $\overline{V}$ that includes the finite set of boundary points $\del V$.  If there exists $m$ so that $f^m(a)$ lies in $P$, then we let $\zeta_B = f^m(\zeta_a)$ and let $n$ be its period.  Note that $f^m(a)$ lies in a periodic $F$-disk for the new $\Gamma'$.  If $\zeta_a$ has infinite forward orbit, let $\zeta_B \in P$ be the Type II point which is the retraction of the iterate $f^m(a)$ to $P$; by increasing $m$ if necessary, we can arrange so that no elements of $\Gamma'$ lie in this component of $V\setminus P$ containing $f^m(a)$ nor in any component containing the forward orbit of $f^m(a)$.    We then include $\zeta_a, f(\zeta_a), \ldots, f^{m-1}(\zeta_a)$ and the orbit of $\zeta_B$ in $\Gamma'$.  Note that $f^m(a)$ now lies in wandering $F$-disk for $\Gamma'$.  The point $\zeta_B$ is $k$-split because the orbit of $a$ intersects infinitely many directions from $\zeta_B$.

In all cases, the pair $(f,\Gamma')$ is analytically stable for the newly augmented $\Gamma'$, and each element of the orbit $f^j(a)$, for $j\geq 0$, must lie in an $F$-disk.  This proves Theorem \ref{stable modification}.  Moreover, in all three cases, if we set $\Gamma_B = \{\zeta_B\}$, then the pair $(f^n, \Gamma_B)$ is also analytically stable, and the point $f^m(a)$ and its future iterates $f^{m+ jn}(a)$, $j\geq 1$, are in $F$-components for $\Gamma_B$.  In particular, they do not lie in hole-directions from $\zeta_B$.  We make any choice of $B \in \mathrm{PGL}_2(k)$ that takes $\zeta_B$ to the Gauss point $\zeta_G$. In view of Lemma \ref{holes are bad} this completes the proof of Theorem \ref{B point}.
\end{proof}

\subsection{Proof of Theorem \ref{modification}}
This is a consequence of Theorem \ref{stable modification}, similar to the proof of \cite[Theorem E]{DF:degenerations2} as an application of \cite[Theorem D]{DF:degenerations2}.  Fix $\gamma \in X(K)$, and assume that $a \in \bP^1(k)$ is in the non-archimedean Fatou set of $f$ at $\gamma$.  Choose a Zariski open set $U_\gamma \subset X$ so that all indeterminacy points of 
	$$\tilde{f}:  U_\gamma \times \bP^1 \dashrightarrow U_\gamma \times \bP^1$$
lie over $\gamma$.  Apply Theorem \ref{stable modification} to $\Gamma = \{\zeta_G\}$, the Gauss point, in $\bP^{1,an}_\gamma$.  The analytically stable pair $(f, \Gamma')$ guaranteed by Theorem \ref{stable modification} gives rise to a birational morphism $Y_\gamma \to U_\gamma \times \bP^1$ defined over $K$, which is an isomorphism outside of $\{\gamma\}\times \bP^1$, and an algebraically stable map $\tilde{f}_\gamma:  Y_\gamma \dashrightarrow Y_\gamma$ lifting $\tilde{f}$. (See \cite[\S4]{DF:degenerations2} for details on the relationship between vertex sets $\Gamma$ and modifications of the surface $X\times \bP^1$.)  Recall that $C_a$ denotes the curve in $X \times \bP^1$ defined the graph of $t \mapsto a_t$, and that $C_{f^n(a)}^{Y_\gamma}$ denotes the proper transform of the curve $C_{f^n(a)}$ in $Y_\gamma$, for each $n\geq 0$. 

Let $\pi: Y_\gamma \to U_\gamma$ denote the projection.  The indeterminacy points for the iterates $(\tilde{f}_\gamma)^n$ in $\pi^{-1}(\gamma)$ are identified with $J$-components of $\Gamma'$, and the $F$-disks for $\Gamma'$ are identified with smooth points in the fiber $\pi^{-1}(\gamma)$ that are not indeterminate for any iterate of $\tilde{f}_\gamma$.  Therefore, the conclusion of Theorem \ref{stable modification} about the Fatou point $a$ guarantees that the curves $C_{f^n(a)}^{Y_\gamma}$ are disjoint from $I(\tilde{f}_\gamma)$ and intersect the fiber over $\gamma$ in smooth points, for all $n\geq 0$.  Assuming that the point $a$ is totally Fatou, we can repeat this argument over each $\gamma \in X(K)$ where $\tilde{f}: X \times \bP^1 \dashrightarrow X \times \bP^1$ has indeterminacy; we glue the surfaces $Y_\gamma$ and maps $f_\gamma$ to obtain our desired rational map 
	$$\tilde{f}_Y : Y \dashrightarrow Y.$$  
	
For the converse implication, let $Y \to X \times \bP^1$ be any choice of birational morphism defined over $K$, and let $\pi: Y \to X$ be the projection to the first factor.  Assume that $a \in \bP^1(k)$ lies in the non-archimedean Julia set at $\gamma \in X(K)$.  Then we know that the curve $C_a^Y$ will intersect an indeterminacy point of some iterate $(\tilde{f}_Y)^j$ in the fiber of $Y$ over $\gamma$, by Lemma \ref{Julia holes}.  Indeed, any small Berkovich disk around $a$ will map, under large iterates of $f$, over each of the Type II points corresponding to the components of $Y$ over $\gamma$.  There are now two cases to consider.  If $C_a^Y$ or some iterate $C_{f^n(a)}^Y$ intersects a component of the fiber over $\gamma$ which is mapped by $\tilde{f}_Y$ into an indeterminacy point, then we are done.  If not, then since the point $p = C_a^Y \cap \pi^{-1}(\gamma)$ is indeterminate for $(\tilde{f}_Y)^j$, it must be that the point $p$ is sent by $(\tilde{f}_Y)^m$, for some $m < j$, to an element of $I(\tilde{f}_Y)$.  Consequently, $C_{f^m(a)}^Y$ intersects the indeterminacy set of $\tilde{f}_Y$, and the proof is complete.   
\qed

\subsection{Proof of Theorem \ref{dense}}
Assume that $f$ is defined over $k = K(X)$, for number field $K$.  Enlarging $K$ if necessary, we can assume that all places $\gamma$ of bad reduction for $f$ lie in $X(K)$.   At each place $\gamma \in X(\Kbar)$ of $k$, we know that the non-archimedean Fatou set $\Omega_\gamma(f) \cap \bP^1(k)$ is open in $\bP^1(k)$, in the $\gamma$-adic topology.  We also know that $\Omega_\gamma(f) \cap \bP^1(k) = \bP^1(k)$ for all but finitely many $\gamma$.  We will show that $\Omega_\gamma(f) \cap \bP^1(k)$ is dense in $\bP^1(k)$ for the remaining $\gamma$.  

Fix $\gamma \in X(K)$ and a point $b \in \bP^1(k)$, and let $\zeta$ be any $k$-split Type II point in $\bP^{1,an}_\gamma$ bounding a disk around $b$.  Consider all the connected components of $\bP^{1,an}_\gamma \setminus \{\zeta\}$.  If one of these disks intersects $\bP^1(k)$, we call it a {\bf $k$-disk} at $\zeta$.   In the natural identification of the set of components of $\bP^{1,an}_\gamma \setminus \{\zeta\}$ with $\bP^1(\Kbar)$, the $k$-disks correspond to the rational points $\bP^1(K)$.  If a $k$-disk at $\zeta$ intersects the Berkovich Julia set, we call it a {\bf Julia $k$-disk} at $\zeta$.  

If there are only finitely many Julia $k$-disks at $\zeta$, then we can always find infinitely many $k$-disks at $\zeta$ that are fully contained in the Fatou set.  This shows the existence of Fatou elements of $\bP^1(k)$ in the closed Berkovich disk around $b$ bounded by $\zeta$.  

If there are infinitely many Julia $k$-disks at $\zeta$, then $\zeta$ is in the Julia set (because the Julia set is closed in $\bP^{1,an}_\gamma$), and $\zeta$ is therefore preperiodic \cite[Proposition 3.9]{DF:degenerations2}.   We are still able to find infinitely many $k$-disks at $\zeta$ that are fully contained in the Fatou set.  Suppose that $f^{m+n}(\zeta) = f^m(\zeta)$ for some $m \geq 0$ and $n\geq 1$.  Let $e$ be the local degree of $f^n$ at $f^m(\zeta)$, as defined in \cite[\S7.4]{Benedetto:book}, so that $e\geq 1$.  For $e > 1$, the iterate $f^n$ induces a map $g: \bP^1 \to \bP^1$ of degree $e$, defined over $K$, by the natural identification of $\bP^1(\Kbar)$ with the set of directions from $f^m(\zeta)$. The Julia set of $f$ (which coincides with the Julia set for $f^n$) in $\bP^{1,an}_\gamma$ is contained in the union of the hole-directions from $f^m(\zeta)$ for $f^n$ and its iterates $f^{jn}$, $j\geq 1$, by Lemma \ref{Julia holes}.  In other words, the Julia directions are identified with a subset of the union $\bigcup_{j\geq 0} g^{-j}(E)$ for a finite set $E \subset \bP^1(\Kbar)$, corresponding to the hole-directions for $f^n$. But this implies that there are only finitely many Julia $k$-disks from $f^m(\zeta)$, because they correspond to a set in $\bP^1(K)$ with bounded Weil height, since $\deg g > 1$.  It follows that there were only finitely many Julia $k$-disks from $\zeta$, a contradiction.  So we conclude that $e=1$.  

The action of $f^n$ at $f^m(\zeta)$ therefore induces an automorphism $A \in \mathrm{PGL}_2(K)$ acting on $\bP^1(\Kbar)$, the set of directions from $f^m(\zeta)$.  The Julia $k$-directions from $f^m(\zeta)$ are contained in the union of the finitely many hole-directions of $f^n$ at $f^m(\zeta)$ and the hole-directions for all iterates $f^{nj}$, $j \geq 1$. As before, these directions are identified with the union of a finite set $E$ in $\bP^1(\Kbar)$ and the orbit of $E$ under $A^{-1}$.  Now let $h: \bP^1 \to \bP^1$ be the map induced by $f^m$ from $\zeta$ to $f^m(\zeta)$, defined over $K$, under any choice of identification of the $k$-directions from $\zeta$ and $f^m(\zeta)$ with $\bP^1(K)$.  Then we can find infinitely many $k$-disks at $\zeta$ that are fully contained in the Fatou set, as a consequence of the following:

\begin{lemma}  
For any number field $K$, any finite set $E$ in $\bP^1(K)$, any $A: \bP^1\to \bP^1$ of degree 1 defined over $K$, and any nonconstant $h: \bP^1\to \bP^1$ defined over $K$, there exists an infinite set $Y \subset \bP^1(K)$ for which 
	$$\left( \bigcup_{j \geq 0} A^j(h(Y))\right) \cap E = \emptyset.$$
\end{lemma}

\begin{proof}  
Choosing coordinates on $\bP^1$ over $K$, we can assume that $A^{-1}(x) = \alpha x$ for some $\alpha \in K^*$ or that $A^{-1}(x) = x+1$.  If $A$ has finite order, then there is nothing to show, as then $\bigcup_{j\geq 0} A^{-j}(E)$ is finite while $h(\bP^1(K))$ is infinite.  So if $A^{-1}(x) = \alpha x$, we can assume there exists a place $v$ of $K$ for which $|\alpha|_v > 1$. Then the set $\bigcup_{j\geq 0} A^{-j}(E) = \bigcup_{j\geq 0} \left(\alpha^j E\right)$ has no $v$-adic accumulation points except $\infty$. Choose any $y_0 \in \bP^1(K)$ so that $h(y_0)\neq \infty$, and let $\{y_n\}$ be any infinite sequence in $\bP^1(K)$ for which $y_n \to y_0$ $v$-adically.  Then $h(y_n) \to h(y_0)$ $v$-adically.  Therefore, letting $Y$ be this sequence $\{y_n\}$, after excluding at most finitely many elements from the sequence, we may conclude that $\left(\bigcup_{j\geq 0} A^{-j}(E) \right) \cap h(Y) = \emptyset$.  For $A^{-1}(x) = x+1$, we work with any archimedean place of $K$.  Let $y_0\in \bP^1(K)$ be a point for which $h(y_0) \not= \infty$, and select any sequence $y_n \in K$ for which $y_n \to y$ at this place.  Then, as before, letting $Y$ be the complement of finitely many points in $\{y_n\}$, we conclude that $\left(\bigcup_{j\geq 0} A^{-j}(E) \right) \cap h(Y) = \emptyset$.
\end{proof}

Repeating the above argument for all $k$-split points $\zeta$, we see that $U_\gamma := \Omega_\gamma(f) \cap \bP^1(k_\gamma)$ is open and dense in $\bP^1(k_\gamma)$ in the $\gamma$-adic topology. 

Let $\gamma_1, \ldots, \gamma_s \in X(K)$ denote the places for which $U_\gamma \not= \bP^1(k_\gamma)$.  Via the canonical embedding of $k$ into $\prod_{\gamma \in X(\Kbar)} k_\gamma$, we can approximate any tuple $(x_1, \ldots, x_s) \in \prod_i U_{\gamma_i}$ by elements in $k$.  This shows that totally Fatou points are open and dense in $\bP^1(k)$ in the topology induced from the product topology.  \qed

\subsection{Intersection theory for a Fatou point}
\label{intersection theory}
The existence of the resolution $Y \to X\times \bP^1$ constructed in Theorems \ref{modification} and \ref{stable modification} shows more. We now prove that the local geometric canonical height $\hat{\lambda}_{f, \gamma}(a)$, at each place $\gamma$ of the function field $k = K(X)$, can be computed as an intersection number in $Y$ when $a \in \bP^1(k)$ is totally Fatou.  In this way, for each place $\gamma$ of $k$, we can view our surface $Y$ as providing a relative type of N\'eron model, associated to the pair $(f,a)$.

Fix a choice of local canonical height functions $\{\hat\lambda_{f,\gamma}: \gamma \in X(\Kbar)\}$ on $\bP^1(k)$ as in \cite[\S3.5]{Silverman:dynamics}, so that $\hat{h}_f(a) = \sum_{\gamma \in X(\Kbar)} \hat{\lambda}_{f,\gamma}(a)$ for every $a \in \bP^1(k)$.  The local canonical height can be computed as 
	$$\hat{\lambda}_{f,\gamma}(a) = - \min\{0, \ord_\gamma(a)\}$$
at all but finitely many places $\gamma$; we enlarge the number field $K$ so that this finite set of places is contained in $X(K)$.  

Recall that, as in the statement of Theorem \ref{modification}, the curve $C_a$ is the section of $X\times \bP^1 \to X$ defined by $t \mapsto a_t$, for any $a \in \bP^1(k)$.  The curve $C_a^Y$ is its proper transform in $Y$.

\begin{proposition} \label{relative Neron}
Let $f: \bP^1 \to \bP^1$ be of degree $d > 1$, defined over a function field $k = K(X)$.  Fix $\gamma\in X(K)$.  Let $\pi:  Y \to X\times \bP^1$ be a birational morphism defined over the number field $K$, which is an isomorphism outside of the line $L_\gamma = \{\gamma\} \times \bP^1$.  Let $\{Y_{\gamma,i}\}_{i=1}^{m_\gamma}$ denote the irreducible components of $E_\gamma = \pi^{-1}(L_\gamma)$.  Let $\tilde{f}_Y$ be the induced map on $Y$ satisfying
$$\xymatrix{ Y \ar[d]_\pi \ar@{-->}[r]^{\tilde{f}_Y} & Y\ar[d]^\pi \\   X\times\bP^1  \ar@{-->}[r]^{\tilde{f}} & X\times\bP^1  }$$
and assume that $\tilde{f}_Y$ maps no component $Y_{\gamma, i}$ into an indeterminacy point of $\tilde{f}_Y$ in $E_\gamma$.  Then, there exist rational numbers $c_{\gamma,i} \in \bQ$, for $i = 1, \ldots, m_{\gamma}$,  so the following holds.  For each point $a \in \bP^1(k)$ such that the curve $C_{f^n(a)}^Y$ is disjoint from the indeterminacy locus $I(\tilde{f}_Y)\cap E_\gamma$ and the singular locus of $E_\gamma$ for every $n\geq 0$, the local geometric canonical height of $a$ at $\gamma$ is computed by 
 $$\hat{\lambda}_{f,\gamma}(a) = (C_a \cdot C_\infty)_\gamma + \sum_{i=1}^{m_{\gamma}}c_{\gamma,i} \; C_a^Y \cdot Y_{\gamma,i}$$
where $(C_a \cdot C_\infty)_\gamma$ is the intersection multiplicity of the curves $C_a$ and $C_\infty$ in $X\times \bP^1$ at $(\gamma, \infty)$. 
\end{proposition}

Combined with Theorems \ref{modification} and \ref{stable modification}, we obtain:

\begin{theorem} \label{intersection computation}
Let $f: \bP^1 \to \bP^1$ be of degree $d > 1$, defined over a function field $k = K(X)$, and let $a \in \bP^1(k)$ be a totally Fatou point.  Extending the number field $K$ if necessary, let $Y$ be the surface of Theorem \ref{modification}, and let $\{c_{\gamma, i}\}$ be the rational numbers guaranteed by Proposition \ref{relative Neron} over each $\gamma \in X(K)$. Then the geometric canonical height of $a$ satisfies 
 $$\hat{h}_f(a) = \sum_{\gamma \in X(\Kbar)} \hat{\lambda}_{f,\gamma}(a) =  C_a \cdot C_\infty + \sum_{\gamma \in X(K)} \sum_{i=1}^{m_{\gamma}} c_{\gamma, i} \; C_a^Y \cdot Y_{\gamma,i}.$$
\end{theorem}

\begin{proof}
The theorem is almost immediate from Proposition \ref{relative Neron} and the statement of Theorem \ref{modification}, summing over all $\gamma \in X(\Kbar)$.  We only need the additional input of Theorem \ref{stable modification} that the orbit of $a$ will always lie in an $F$-disk for the vertex set $\Gamma'$.  This guarantees that the curves $C_{f^n(a)}^Y$ intersect the singular fibers only in their smooth points.  
\end{proof}

\begin{proof}[Proof of Proposition \ref{relative Neron}]
As for Theorem \ref{stable modification}, we continue to follow the arguments of \cite{DF:degenerations2}, and we also build on the machinery developed in \cite{DG:rationality}.   

We identify the components $Y_{\gamma,i}$ with a finite set of Type II points in the Berkovich space $\bP^{1,an}_\gamma$ over the field $\bL_\gamma$.  (We caution that some components $Y_i$ may be non-reduced, so we need to keep track of their multiplicities as well.)  See the discussion in, e.g., \cite[\S4]{DF:degenerations2}.  Let $\Gamma \subset \bP^{1,an}_\gamma$ be the union of this finite set of Type II points; note that $\Gamma$ must include the Gauss point of $\bP^{1,an}_\gamma$ because $\pi: Y \to X \times \bP^1$ is regular.

Suppose that $a\in \bP^1(k)$ is a point for which the curves $C^Y_{f^n(a)}$ are disjoint from the points of indeterminacy for $\tilde{f}_Y$ for all $n\geq 0$.  This means, as in \S\ref{Gamma stuff}, that $f^n(a)$ lies in an $F$-component for $\Gamma$ for every $n\geq 0$.  Fixing a homogeneous lift $F$ of $f$ so that $\ord_\gamma F = 0$, we define the order function $\sigma(F, \cdot)$ on $\bP^{1,an}_\gamma$ as in \cite[\S3.1]{DG:rationality}.  Specifically, for each $n\geq 0$, we let $A_n$ denote a homogeneous lift of $f^n(a) \in \bP^1(k)$ so that $\ord_\gamma A_n = 0$, and then
	$$\sigma_n := \sigma(F, f^n(a)) = \ord_\gamma F(A_n).$$
From \cite[Lemma 3.1]{DG:rationality}, the local canonical height at $\gamma$ (associated to this choice of $F$) can be computed as
	$$\hat{\lambda}_{f,\gamma}(a) = - \min\{0, \ord_\gamma(a)\} - \sum_{n = 0}^\infty \frac{\sigma_n}{d^{n+1}}.$$
The key observation is contained in \cite[Proposition 4.1, Theorem 4.2]{DG:rationality}:  for a point $a$ that lies in an $F$-disk component of $\bP^{1,an}_\gamma\setminus \Gamma$, the order function depends only on the boundary point of that $F$-disk.  When each iterate of $a$ lies in an $F$-disk, the sequence $\sigma_n$ depends only on the sequence of boundary points of these $F$-disks containing $f^n(a)$, over all $n\geq 0$.  However, by the stability of the pair $(f, \Gamma)$, these sequences, in turn, depend only on the boundary point of the disk containing $a$ itself.  Indeed, every $F$-disk with boundary point $\zeta$ will map into an $F$-disk with the same boundary point.  Moreover, the order function can only take finitely many possible values on the $F$-disks of $\Gamma$ (by \cite[Theorem 4.2]{DG:rationality}) and the stability of $(f,\Gamma)$ implies that the sequence $\{\sigma_n\}$ will be eventually periodic.  In other words, the sequence $\{\sigma_n\}$ depends only on the component $Y_{\gamma,i}$ that intersects $C_a^Y$ in $E_\gamma$.  The coefficient $c_{\gamma,i}$ is rational because the sequence $\{\sigma_n\}$ is eventually periodic.
\end{proof}

\bigskip
\section{Near a singularity:  uniform convergence to the escape rate}
\label{continuity}

In this section, we fix $\gamma \in X(K)$.  We assume that we are given $f: \bP^1\to \bP^1$ of degree $d \geq 2$, defined over $k = K(X)$, and a point $a \in \bP^1(k)$. We choose lifts $F$ and $A$, as defined in \S\ref{lifts},with $\ord_\gamma F = \ord_\gamma A = 0$ and we assume that 
	$$\Res(F_\gamma) = 0.$$
We also assume that the pair $(f,a)$ is hole-avoiding at $\gamma$, as defined in \S\ref{hole-avoiding}, so that 
	$$F_\gamma^n(A_\gamma) \not= (0,0)$$
for all $n\geq 0$.   We set
	$$A_n := F^n(A) \in k^2$$
and we study the convergence of the sequence of functions 
\begin{equation} \label{g_n}
	g_{n,v}(t) := \frac{1}{d^n} \log\|(A_n)_t\|_v
\end{equation}
in a neighborhood of $t=\gamma$ in $X(\bC_v)$, for each $v\in M_K$.  We prove:

\begin{theorem}  \label{uniform convergence}
Fix $\gamma \in X(K)$ and a hole-avoiding pair $(f,a)$ at $\gamma$ with lifts $F$ and $A$ satisfying $\ord_\gamma F = \ord_\gamma A = 0$ and $\Res(F_\gamma) = 0$.  There exists an $M_K$-neighborhood $\mathfrak{U}$ of $\gamma$ in $X$ so that, for each $v \in M_K$, the functions $g_{n,v}$ converge uniformly on $\mathfrak{U}_v$ to a continuous function $g_v$.  
\end{theorem}

Note that the limit function $g_v$ coincides with the escape-rate function $G_{F_t,v}(A_t)$ defined by \eqref{arithmetic escape} in \S\ref{aer}, for $t\not=\gamma$.  So we know that the convergence of $g_{n,v}$ to $g_v$ is uniform on neighborhoods where $t$ remains {\em bounded away} from $\gamma$ and the other singularities of $f$.   The steps in the proof of Theorem \ref{uniform convergence} are inspired by the arguments in \cite{DWY:Lattes}, \cite{DWY:QPer1}, and \cite{Mavraki:Ye}.

\subsection{Convergence of the constant terms $g_{n,v}(\gamma)$}

\begin{proposition}  \label{alpha}
Fix $\gamma \in X(K)$.  Under the hypotheses of Theorem \ref{uniform convergence}, the limit 
	$$\alpha_v := \lim_{n\to\infty} \frac{1}{d^n} \log\|(A_n)_\gamma\|_v$$
exists in $\R$, for all $v\in M_K$.  Moreover, we have 
	$$\sum_{v\in M_K} N_v |\alpha_v| < \infty.$$
In other words, $\{\alpha_v: v \in M_K\}$ defines an $M_K$-quasiconstant.
\end{proposition}

\begin{remark}  \label{quasi adelic}
For each fixed $n$, we have $\|(A_n)_\gamma\|_v = 1$ for all but finitely many $v$.  But as $n$ grows, the number of places for which $\|(A_n)_\gamma\|_v \not=1$ can also grow, so that $\alpha_v$ can be nonzero for infinitely many $v \in M_K$.  A simple example is given by the function $f(z) = z(z+1)/(z+t)$ defined over $k = \bQ(t)$, which is similar to Example \ref{translation example}, at $t=0$.  Take $a = 1$.  Fix homogeneous polynomial lift $F(z,w) = (z(z+w), (z+tw)w)$, so that $F_0(z,w) = (z(z+w), zw)$, and set $A = A_0 = (1,1)$.  Then for every prime $p$, we have $\|(A_n)_0\|_p = 1$ for all $n< p$, and $\|(A_n)_0\|_p < 1$ for all $n\geq p$. We show below in \eqref{decreasing} that this will imply that the limit $\alpha_p$ of Proposition \ref{alpha} will be negative for all primes $p$. Many more examples are given in \cite{Mavraki:Ye}.  Bear in mind that this does not happen for Latt\`es examples (the maps arising as quotients of endomorphisms of elliptic curves) or for polynomials; in other words, for those types of maps, the $\alpha_v$ of Proposition \ref{alpha} always define an $M_K$-constant.  
\end{remark}

\begin{proof}[Proof of Proposition \ref{alpha}]
Since $\Res(F_\gamma) = 0$, specializing $F$ at $\gamma$, we can write 
	$$F_\gamma = H \, \hat F$$
where $H(z,w) \in K[z,w]$ is a nonconstant homogeneous polynomial of degree $k \leq d$, and $\hat F (z,w) \in (K[z,w])^2$ is a homogeneous polynomial map of degree $\ell = d - k < d$ inducing a morphism of degree $\ell$ on $\bP^1$.   The zeroes of $H$ in $\bP^1$ are called the holes of $f$ at $\gamma$, as defined in \S\ref{hole-avoiding}.  

Because the pair $(f,a)$ is hole-avoiding, the lift $A$ satisfies $F^n_\gamma(A_\gamma) \not= (0,0)$ for all $n$.  So it must be that either $\ell >0$ or, if $\ell =0$, the value of $\hat F$ is not a root of $H$.  Consequently, as in \cite[Lemma 2.2]{D:measures}, the specialization of each iterate $F^n$ can be expressed in terms of $H$ and $\hat F$ by
\begin{equation} \label{F0 iterates}
	F_\gamma^n(z,w) = \left(\prod_{i=0}^{n-1} H(\hat F^i(z,w))^{d^{n-1-i}} \right)  \hat F^n(z,w)
\end{equation}
for all $n\geq 1$.  In particular, this shows that 
\begin{equation} \label{An0}
	(A_n)_\gamma = \left(\prod_{i=0}^{n-1} H(\hat F^i(A_\gamma))^{d^{n-1-i}} \right)  \hat F^n(A_\gamma)
\end{equation}
for every $n$.  

For $\ell = 0$, the map $\hat F$ is constant, so $\hat F^n(A_\gamma) = (z_0, w_0) \in K^2 \setminus\{(0,0)\}$ for some point $(z_0,w_0)$ and for all $n\geq 1$.  The formula \eqref{An0} gives 
\begin{eqnarray*} 
\frac{1}{d^n} \log\|(A_n)_\gamma\|_v &=& \frac1d \log|H(A_\gamma)|_v +  \sum_{i=1}^{n-1} \frac{1}{d^{i+1}} \log|H(z_0,w_0)|_v + \frac{1}{d^n} \log\|(z_0,w_0)\|_v \\
&\longrightarrow&  \frac1d \log|H(A_\gamma)|_v+  \frac{1}{d(d-1)} \log|H(z_0,w_0)|_v
\end{eqnarray*}
as $n\to\infty$, for all places $v$ of $K$.  The statements of the proposition follow immediately in this case.

Now assume that $\ell \geq 1$.  There exists an $M_K$-constant $\mathfrak{L}$ so that 
	$$e^{-\mathfrak{L}_v} \|(z,w)\|_v^\ell \leq \|\hat F(z,w)\| \leq e^{\mathfrak{L}_v} \|(z,w)\|_v^\ell$$
for all $(z,w) \in \Kbar^2$ and for all $v \in M_K$ \cite[Proposition 5.57]{Silverman:dynamics}.  This implies that 
\begin{equation}  \label{An0 F hat}
	e^{-\mathfrak{L}_v(1+\ell+\cdots+\ell^{n-1})} \|A_\gamma\|_v^{\ell^n} \leq \|\hat F^n(A_\gamma)\|_v \leq 
	e^{\mathfrak{L}_v(1+\ell+\cdots+\ell^{n-1})} \|A_\gamma\|_v^{\ell^n}
\end{equation}
so that
\begin{equation} \label{easy part}
	\lim_{n\to\infty} \frac{1}{d^n} \log\|\hat F^n(A_\gamma)\|_v = 0
\end{equation}
for all $v \in M_K$, because $\ell < d$. 

Recalling that $\deg H = k = d - \ell$, there is also an $M_K$-constant $\mathfrak{H}$ so that 
	$$|H(z,w)|_v \leq e^{\mathfrak{H}_v} \| (z,w)\|^k_v$$
for all $(z,w) \in K^2$.  So
\begin{equation} \label{upper bound on H}
	|H(\hat F^i(A_\gamma))|_v \leq e^{\mathfrak{H}_v} \|\hat F^i(A_\gamma)\|_v^k \leq e^{\mathfrak{H}_v}e^{k\mathfrak{L}_v(1+\ell+\cdots+\ell^{i-1})} \|A_\gamma\|_v^{k\ell^i}
\end{equation}
at all places $v$ and for all $i\geq 1$.  Note that the bound on the right side of \eqref{upper bound on H} can be $>1$ at only finitely many places $v$ of $K$, independent of $i$.  Let $S_+$ denote this finite set of places. Therefore, since $H(\hat F^i(A_\gamma)) \in K^*$ for all $i$, we can apply the product formula to observe that there is a constant $c > 0$ so that 
\begin{equation} \label{hole part}
	\prod_{v \not\in S_+ }  |H(\hat F^i(A_\gamma))|_v^{N_v} \geq c^{\max\{i, \ell^i\}} 
\end{equation}
for all $i\geq 1$.  Using the formula \eqref{An0}, we combine \eqref{easy part} with \eqref{upper bound on H} and \eqref{hole part} to deduce the existence of
\begin{equation} \label{alpha as a sum}
\alpha_v = \lim_{n\to\infty} \frac{1}{d^n} \log\|(A_n)_\gamma\|_v =  \sum_{i=0}^\infty \frac{1}{d^{i-1}} \log|H(\hat F^i(A_\gamma))|_v
\end{equation}
at every place $v$, because $\ell < d$.

From \eqref{upper bound on H} and the summation expression for $\alpha_v$ in \eqref{alpha as a sum}, we see that $\alpha_v \leq 0$ for all $v \not\in S_+$.  To show that the sum over all places of the $\alpha_v$ is finite, we use \eqref{hole part} to estimate
	$$\frac{1}{d^{i-1}} \sum_{v \not\in S_+} N_v \log |H(\hat F^i(A_\gamma))|_v \geq \frac{\max\{i, \ell^i\}}{d^{i-1}} \log c$$
for each $i\geq 1$.  Summing over all $i$, we can then use Fubini's theorem to deduce that
	$$\sum_{v \not\in S_+} N_v \, \alpha_v  > -\infty,$$  
so that 
	$$\sum_{v\in M_K} N_v |\alpha_v| < \infty.$$
This completes the proof of the proposition.
\end{proof}

\subsection{Proof of Theorem \ref{uniform convergence}}  
Throughout this proof, we work in an $M_K$-neighborhood $\mathfrak{U}$ of $\gamma \in X(K)$, so that the conclusion of Proposition \ref{F v bounds} holds.  For simplicity, we let $u \in K(X)$ denote a choice of local coordinate on $X$ near $\gamma$ so that $u=0$ represents $\gamma$. 

We now fix $v \in M_K$, and we drop the dependence on $v$ to ease notation. Let $\delta$ denote the $v$-adic radius of the largest disk $\{|u|_v < \delta\}$ contained in the $M_K$-neighborhood $\mathfrak{U}_v$. Let $C = e^{\mathfrak{b}_v} \geq  1$ be the constant appearing in Proposition \ref{F v bounds} at this place.  For each $n$, we write $A_n(u)$ for the specialization of $A_n = F^n(A)$ at $u$. For every $n\geq m$, we define 
\begin{eqnarray*} 
g_n(u) &=&  \frac{1}{d^n}\log\|A_n(u)\| \\
	&=&  \frac{1}{d^m}\log\|A_m(u)\| + \frac{1}{d^m} \sum_{j = 1}^{n-m} \frac{1}{d^j} \log \frac{\|A_{m+j}(u)\|}{\|A_{m+j-1}(u)\|^d} 
\end{eqnarray*}
for $|u|<\delta$. Let $q = \ord_\gamma \Res(F)$.  
From Proposition \ref{F v bounds}, we have
\begin{equation} \label{basic}
g_m(u) + \frac{1}{d^m} \left( q \log|u| - \log C \right) ~ \leq  ~ g_n(u) ~ \leq ~ g_m(u) + \frac{1}{d^m}\log C
\end{equation}
for all $n\geq m \geq 0$ and for all $|u| < \delta$.  Let 
	$$\alpha = \lim_{n\to\infty} \frac{1}{d^n} \log\|A_n(0)\|;$$
its existence is guaranteed by Proposition \ref{alpha}.  

\smallskip\noindent
{\bf Step 1: a choice of $N$ and $\delta_N$ for a uniform upper bound.}  Fix $\eps>0$.  Choose $N$ so that we have 
\begin{equation} \label{large n}
e^{d^n(\alpha-\eps)} \leq \|A_n(0)\|  \leq  e^{d^n(\alpha+\eps)}
\end{equation} 
for all $n\geq N$, and so that
\begin{equation} \label{N for C}
 \frac{1}{d^N} \log C < \eps \quad\mbox{ and } \quad \left| \frac{1}{d^N} \log(1-\eps) \right| < \eps.
\end{equation}

Now choose $\delta_N > 0 $ so that, by continuity of $A_N(u)$, we have 
\begin{equation} \label{A_N control} 
	(1-\eps) e^{d^N(\alpha-\eps)} \leq \|A_N(u)\|  \leq e^{d^N(\alpha+2\eps)}
\end{equation}
for all $|u| \leq \delta_N$.  Applying the upper bound of \eqref{basic} and using \eqref{N for C}, this implies that
\begin{equation} \label{upper bound delta_N}
	g_n(u) ~ \leq ~  g_N(u) + \frac{1}{d^N} \log C  ~ \leq ~   \alpha + 3\eps
\end{equation}
for all $n\geq N$ and for all $|u|\leq \delta_N$.   

Note that the lower bound of \eqref{basic} is not enough to get uniform control on $g_n$ from below for $n\geq N$, because of the $\log|u|$ term.

\smallskip\noindent
{\bf Step 2:  the Maximum Principle and lower bounds within $\delta_n$.}  By the triangle inequality, we have 
	$$\|A_N(u) - A_N(0)\| \leq 2e^{d^N(\alpha+2\eps)}$$
for all $|u|\leq \delta_N$, from \eqref{large n} and \eqref{A_N control}.  Note that the coordinates of $A_N(u) - A_N(0)$ vanish at $t=0$, and so the Maximum Principle (applied to $\frac1u (A_N(u) - A_N(0))$) gives 
	$$\|A_N(u) - A_N(0)\| ~ \leq ~ \frac{|u|}{\delta_N} \; 2e^{d^N(\alpha+2\eps)}$$
for all $|u| \leq \delta_N$.  For a non-archimedean Maximum Principle see e.g. \cite[Proposition 8.14]{BRbook}. Using the upper bound of \eqref{upper bound delta_N}, the same argument implies that 
\begin{equation}  \label{max princ}
	\|A_n(u) - A_n(0)\| ~ \leq ~ \frac{|u|}{\delta_N} \; 2e^{d^n(\alpha+3\eps)}
\end{equation}
for all $n\geq N$ and for all $|u| \leq \delta_N$.  This implies that 
\begin{eqnarray*} 
\|A_n(u) \| &\geq& \|A_n(0)\| - \|A_n(u) - A_n(0)\| \\
	&\geq&  e^{d^n(\alpha-\eps)} - \frac{|u|}{\delta_N} \; 2e^{d^n(\alpha+3\eps)} \\
	&=& e^{d^n(\alpha-\eps)}\left( 1 - \frac{|u|}{\delta_N} \; 2e^{d^n(4\eps)} \right) 
\end{eqnarray*}
for all $n \geq N$ and for all $|u| \leq \delta_N$.  

Now define 
\begin{equation} \label{delta_n}
	\delta_n :=   \frac{\delta_N \eps}{2 e^{4d^n\eps}}
\end{equation}
for all $n > N$.  So we have 
\begin{equation} \label{A_n for delta_n}
\|A_n(u) \| ~  \geq ~  e^{d^n(\alpha-\eps)}\left( 1 - \eps \right) 
\end{equation}  
for all $|u| \leq \delta_n$ and for all $n > N$.  Combined with the lower bound of \eqref{A_N control} and the condition on $N$ in \eqref{N for C}, this shows that 
\begin{equation} \label{g_n from below}
	g_n(u) \geq  \alpha - 2\eps
\end{equation}
for all $|u| \leq \delta_n$ and for all $n\geq N$.  

\smallskip\noindent
{\bf Step 3:  Choosing larger $N_0$ and completing the proof.}  From the definition of $\delta_n$, we see that
	$$\frac{1}{d^n} \log \delta_n = \frac{1}{d^n} \log(\delta_N \eps/2) - 4 \eps$$
for all $n > N$.  Now choose $n_0 > N$ so that 
	$$\left | \frac{1}{d^{n_0}} \log(\delta_N \eps/2)  \right|  < \eps.$$
Recall that the sequence $\{g_n\}$ converges uniformly on neighborhoods in $u$ that are bounded away from $u=0$ (and any other singularities for $f$ in $X$), so, by our choice of $M_K$-neighborhood, there exists $N_0 \geq n_0$ so that 
	$$| g_n - g_m| < \eps$$
for all $n, m \geq N_0$, uniformly on $\{ \delta_{n_0} \leq |u| < \delta\}$.

For $|u| \leq \delta_{n_0}$, we know that 
	$$g_n(u) \leq \alpha + 3\eps$$
for all $n\geq n_0$ by \eqref{upper bound delta_N}.  And we know that 
	$$g_n(u) \geq  \alpha - 2\eps$$ 
for all $|u| \leq \delta_n$ and for all $n\geq n_0$, by \eqref{g_n from below}.  On the other hand, for $\delta_n < |u| \leq \delta_{n_0}$ we can choose $n >m \ge n_0$ so that $\delta_{m+1} \leq |u| \leq \delta_m$ and then \eqref{basic} gives 
	$$g_n(u) \geq g_m(u) + \frac{q}{d^m} \log\delta_{m+1} - \frac{1}{d^m} \log C \geq \alpha - 2\eps - 5qd \eps - \eps $$
So, in particular, we have 
	$$\alpha - (3 + 5qd) \eps \leq g_n \leq \alpha + 3\eps$$
for all $n\geq N_0$ and for all $|u| \leq \delta_{n_0}$.  This completes the proof of uniform convergence.  
\qed

\subsection{A summable lower bound on a disk}  
We conclude this section with a consequence of Proposition \ref{alpha} and its proof that will be used to prove Theorem \ref{dyn VCH}.

\begin{proposition} \label{good bound on g_v}
Fix $\gamma\in X(K)$. Under the hypotheses of Theorem \ref{uniform convergence} and in the notation of Proposition \ref{alpha}, there exists a finite set $S_\gamma \subset M_K$ so that
	$$(d (\ord_\gamma \Res F) +1) \, \alpha_v \leq g_v(t) \leq 0$$
for every $v \not\in S_\gamma$ and all $t$ in an $M_K$-neighborhood of $\gamma$.
\end{proposition}

\begin{proof}
We first let $S_\gamma$ be the finite set of places $v \in M_K$, including all archimedean places, at which the quantities $\mathfrak{L}_v$ and $\mathfrak{H}_v$ in the proof of Proposition \ref{alpha} differ from 0 and where $\|A_\gamma\|_v\not=1$.  It follows from the computations in Proposition \ref{alpha} (specifically, equation \eqref{upper bound on H} and \eqref{alpha as a sum}) that $\alpha_v \leq 0$ for all $v \not\in S_\gamma$.  

Recall the formula for $(A_n)_\gamma$ given in \eqref{An0}.  For all $v \not\in S_\gamma$, we have 
	$$\|(A_{n+1})_\gamma\|_v =\|F_\gamma((A_n)_\gamma)\|_v \leq  \|(A_n)_\gamma\|^d$$
so that
\begin{equation} \label{decreasing}
	\frac{1}{d^{n+1}} \log \|(A_{n+1})_\gamma\|_v \leq \frac{1}{d^n} \log \|(A_n)_\gamma\|_v
\end{equation}
for all $n$.  For all $v \in M_K\setminus S_\gamma$, we also have $\|\hat{F}^n(A_\gamma)\|_v = 1$ for all $n$.  So $\|(A_n)_\gamma\|_v < 1$ for some $n>0$ if and only if there exists $i < n$ so that $|H(\hat{F}^i(A_\gamma))|_v < 1$.  Furthermore, from \eqref{decreasing}, such an $n$ exists if and only if $\alpha_v < 0$, for each $v \not\in S_\gamma$.   

Now let $u$ denote a local coordinate on $X$ with $u=0$ representing $\gamma$.
From Proposition \ref{global bound}, the coefficients of $F$ and the coordinates of $A$ are $M_K$-bounded on an $M_K$-neighborhood of $\gamma$.  We enlarge $S_\gamma$ if needed to assume that these coefficients are $\leq 1$ in absolute value and so that the neighborhood is given by $\{|u|_v < 1\}$ for all $v \not\in S_\gamma$.  We further enlarge $S_\gamma$ to include all places at which the $M_K$-constant $\mathfrak{b}_v$ from Proposition \ref{F v bounds} differs from 1, and also so that $|u|_v^{2g+1} = |\xi^\gamma|_v$ for $v\notin S_{\gamma}$ on the $M_K$-neighborhood of $\gamma$ (applying Proposition \ref{global bound} to $u$).

Then, for all $v\not\in S_\gamma$, the upper bound on the coefficients of $F$ and the coordinates of $A$ gives 
\begin{equation} \label{bounded by 1}
	\|A_n(u)\|_v \leq 1
\end{equation}
for all $n \geq 1$ and for all $|u|_v < 1$.  This implies immediately that $g_v(u) \leq 0$ for all $|u|_v < 1$ with $v \not\in S_\gamma$, proving the desired upper bound of the proposition.

Moreover, for $v \not\in S_\gamma$ where $\alpha_v=g_v(0)=0$, we conclude from the Maximum Principle (applied to the subharmonic $g_v$) that $g_v(u) = 0$ for all $|u|_v < 1$, and the estimate of the proposition holds for these $v$.
For the rest of the proof, we fix $v \not\in S_\gamma$ with $\alpha_v < 0$, and choose minimal $m \geq 0$ so that $\|(A_{m+1})_\gamma\|_v < 1$.  Since $\|(A_n)_\gamma\|_v^{1/d^n}$ is a non-increasing sequence from \eqref{decreasing}, we see that $\|(A_n)_\gamma\|_v$ decreases to $0$ as $n\to\infty$, and 
\begin{equation}  \label{above alpha}
	\frac{1}{d^n} \log\|(A_n)_\gamma\|_v \geq \alpha_v
\end{equation}
for all $n$.    As $\|(A_m)_\gamma\|_v = 1$, the inequality \eqref{bounded by 1} implies (with the Maximum Principle) that $\|A_m(u)\|_v = 1$ for all $|u|_v < 1$.

Let $q = \ord_\gamma (\Res F)$.  Proposition \ref{F v bounds} then gives
\begin{equation} \label{t bound}
	\frac{1}{d^n} \log\|A_n(u)\|_v \geq \frac{1}{d^m} \log\|A_m(u)\|_v + \frac{q}{d^m} \log|u|_v =  \frac{q}{d^m} \log|u|_v
\end{equation}
for all $n \geq m$ and for all $|u|_v < 1$.  Therefore, for all $u$ satisfying $\|A_{m+1}(0)\|_v \leq |u|_v < 1$, we have
	$$\frac{1}{d^n} \log\|A_n(u)\|_v ~ \geq ~  \frac{q}{d^m} \log \|A_{m+1}(0)\|_v ~ \geq ~ qd \,\alpha_v$$
for all $n\geq m$, from \eqref{t bound} and \eqref{above alpha}.  This shows that $g_v(u) \geq \, qd\, \alpha_v$ where $|u|_v \geq \|A_{m+1}(0)\|_v$.

On the other hand, for $|u|_v < \|A_{m+1}(0)\|_v$, we can choose $j\geq m+1$ so that 
	$$\|A_{j+1}(0)\|_v \leq |u|_v < \|A_j(0)\|.$$
Then, writing $A_j(u) = A_j(0) + u R_j(u)$ for $u$ near $0$, we know that $\|R_j(u)\|_v \leq 1$ for all $|u|_v < 1$, by \eqref{bounded by 1} and the Maximum Principle, and therefore 
	$$\|A_j(u) \|_v = \|A_j(0) + u R_j(u)\|_v = \|A_j(0)\|_v$$
where $|u|_v < \|A_j(0)\|_v$.  Therefore, 
\begin{eqnarray*}
\frac{1}{d^n} \log\|A_n(u)\|_v &\geq& \frac{1}{d^j} \log\|A_j(u)\|_v + \frac{q}{d^j} \log|u|_v \\
	&\geq&  \frac{1}{d^j} \log\|A_j(0)\|_v + \frac{q}{d^j} \log\|A_{j+1}(0)\|_v \\
	&\geq& (1 + dq)\alpha_v
\end{eqnarray*}
for all $\|A_{j+1}(0)\|_v \leq |u|_v < \|A_j(0)\|$ and for all $n\geq j$.  This implies that $g_v(u) \geq (1 + dq)\alpha_v$ for $u$ values in this region.  Since $\|A_n(0)\|_v \to 0$ as $n\to\infty$, the proof of the lower bound on $g_v$ is complete, thus completing the proof of the proposition.
\end{proof}

\bigskip
\section{Proofs of Theorems \ref{dyn VCH} and \ref{local dyn VCH}}

In this section, we complete the proofs of Theorems \ref{dyn VCH} and \ref{local dyn VCH}.  We fix $f: \bP^1 \to \bP^1$ defined over the field $k = K(X)$, of degree $d>1$, and we assume that $a \in \bP^1(k)$ is totally Fatou for $f$.  Fix homogeneous lifts $F$ of $f$ and $A$ of $a$ as in \S\ref{lifts}.  Recall the definitions of the finite sets $\cS(F) \subset \cS(F,A)$ in $X(\Kbar)$ from \S\ref{lifts}, and that $K$ was enlarged (if necessary) so that $\cS(F,A) \subset X(K)$, as stated in Convention \ref{K choice}. Recall also the definitions of the escape rates $G_{F,\gamma}(A)$ and $G_{F_t, v}(A_t)$ given in \eqref{geometric escape} and \eqref{arithmetic escape}, respectively.  The divisor
	$$D = \sum_{\gamma \in X(K)} G_{F,\gamma}(A)  \; \gamma$$
on $X$ was defined in \eqref{D(F,A)}; its support lies in $\cS(F,A)$.  

For the choice of Weil height $h_D$ defined by \eqref{Weil height}, we now consider the difference 
\begin{multline*}
\hat{h}_{f_t}(a_t) - h_D(t)  = \\
\frac{1}{[K:\bQ]} \frac{1}{|\Gal(\Kbar/K)\cdot t|} \sum_{x \in \Gal(\Kbar/K)\cdot t} \sum_{v \in M_K}  N_v \left(  G_{F_x,v}(A_x) -  \frac{1}{2g+1} \sum_{\gamma \in X(K)} G_{F,\gamma}(A) \log^+|\xi^{\gamma}(x)|_v \right) 
\end{multline*} 
for $t \in X(\Kbar) \setminus \cS(F,A)$. For each place $v$ of $K$, we examine the function 
\begin{equation} \label{V at v}
V_v(t) =  G_{F_t,v}(A_t) -  \frac{1}{2g+1} \sum_{\gamma \in X(K)} G_{F,\gamma}(A) \log^+|\xi^\gamma(t)|_v
\end{equation}
on $X(\Kbar)$ and its extension to $X(\bC_v)$ and the Berkovich analytification $X^{an}_v$. Recall that the steps needed to complete the proofs were outlined in \S\ref{proof goals}.

\subsection{Changing coordinates and lifts} 
If we change the lifts $F$ and $A$, multiplying each by an element of $k^*$, it follows from \eqref{geometric change} and \eqref{arithmetic change} that
\begin{eqnarray} \label{adjustment}
G_{cF_t,v}(\alpha A_t) -  \frac{1}{2g+1} G_{c F,\gamma}(\alpha A) \log^+|\xi^\gamma(t)|_v &=&  
G_{F_t,v}(A_t) -  \frac{1}{2g+1} G_{F,\gamma}(A) \log^+|\xi^\gamma(t)|_v  \nonumber \\
&&  + ~ \frac{1}{d-1}\log|c_t \alpha_t^{d-1}|_v  +  \\
&& \frac{1}{(2g+1)(d-1)} \ord_\gamma (c\alpha^{d-1}) \log^+|\xi^\gamma(t)|_v \nonumber
\end{eqnarray}
for any choice of $\gamma \in X(K)$.  Moreover, the sum of the last two terms is $M_K$-bounded on an $M_K$-neighborhood of $\gamma$, as a consequence of Proposition \ref{global bound}.

If we conjugate $F$ by an element $B \in \mathrm{GL}_2(k)$, we have 
\begin{equation} \label{conjugation} 
		G_{BFB^{-1},\gamma}(B(A))  = G_{F,\gamma}(A) ~ \mbox{ and } ~
G_{(BFB^{-1})_t,v}(B(A)_t)  = G_{F_t,v}(A_t)
\end{equation}
from the definitions of the escape rates, for each $\gamma \in X(K)$, and each place $v$ of $K$ and all $t \in X(\Kbar) \setminus \big(\cS(F,A)\cup \cS(BFB^{-1}, B(A)) \cup \cS(B)\big)$.  Replacing $F$ or $A$ by an iterate gives 
\begin{equation} \label{iterate}
	G_{F^n,\gamma}(F^m(A)) = d^m \, G_{F,\gamma}(A) ~ \mbox{ and } ~ G_{F_t^n,v}(F^m(A)_t) = d^m\,  G_{F_t,v}(A_t)
\end{equation}
for all $n \geq 1$ and $m \geq 0$, again immediate from the definitions.

\subsection{The divisor $D=D(F,A)$ is a $\bQ$-divisor} \label{Q divisor}
We need to show that $G_{F,\gamma}(A) \in \bQ$ for each $\gamma \in \cS(F,A)$.  This is immediate from the following proposition. (It also follows from the statement of Proposition \ref{relative Neron}.) We present an alternative short argument in the following proposition. Recall that $k_\gamma$ denotes the completion of $k$ at $\gamma$.

\begin{proposition} \label{rational}
Let $f: \bP^1\to \bP^1$ be of degree $d \geq 2$, defined over $k = K(X)$, $\gamma$ a point in $X(K)$, and $a\in \bP^1(k_\gamma)$. If the point $a$ is an element of the non-archimedean Fatou set $\Omega_\gamma(f)$ at $\gamma$, then the geometric escape rate $G_{F,\gamma}(A)$ is a rational number, for any choice of lifts $F$ and $A$. 
\end{proposition}

\begin{remark}  In \cite{DG:rationality} it was shown, for maps $f$ defined over $k$, that there can exist points $a \in \bP^1(k_\gamma)$ with irrational local canonical height.  Proposition \ref{rational} implies that these points must always lie in the non-archimedean Julia set of $f$ at $\gamma$.  We provide examples in Section \ref{example section}.  It is not known if the Julia points can be algebraic over $k$.  
\end{remark}

\begin{proof}
If the pair $(f,a)$ is hole-avoiding at $\gamma$, as defined in \S\ref{hole-avoiding}, and if $\beta F$ and $\alpha A$ are lifts of $f$ and $a$, respectively, so that $\ord_\gamma \beta F = \ord_\gamma \alpha A = 0$ for $\alpha, \beta \in k^*$, then 
	$$G_{F,\gamma}(A) = G_{\beta F, \gamma}( \alpha A) + \frac{1}{d-1} \ord_\gamma \beta + \ord_\gamma \alpha$$
from \eqref{geometric change}, so that 
	$$G_{F,\gamma}(A) = 0 + \frac{1}{d-1} \ord_\gamma \beta + \ord_\gamma \alpha  ~ \in ~ \bQ$$
because $\ord_\gamma (\beta F)^n(\alpha A) = 0$ for all $n \geq 0$.

If the pair $(f,a)$ is not hole-avoiding at $\gamma$, then Theorem \ref{B point} implies the existence of a change of coordinates $B \in \mathrm{GL}_2(k)$ and iterates so that the pair $(Bf^nB^{-1}, B(f^m(a)))$ is hole-avoiding at $\gamma$.   The conclusion then follows from \eqref{conjugation} and \eqref{iterate}. 
\end{proof}

\subsection{Variation of canonical height:  proofs of the main theorems}   \label{final}
Assume that $a \in \bP^1(k)$ is totally Fatou for $f$, and let $D = D(F,A)$.  Proposition \ref{rational} implies that $D$ is a $\bQ$-divisor, so it remains to study properties of the functions $V_v$, defined in \eqref{V at v}, associated to this divsor $D$ on the curve $X$ at each place $v$ of the number field $K$.  

We begin by proving Theorem \ref{local dyn VCH}, which states that the functions $V_v$ are continuous on the Berkovich analytification $X^{an}_v$ at all places $v$.  This implies, in particular, the existence of a uniform bound $C_v$ so that $|V_v| \leq C_v$ at all points of $X(\Kbar)$.  (Recall that we have fixed an embedding of $\Kbar \hookrightarrow \bC_v$ for each place $v$.)  Towards proving Theorem \ref{dyn VCH}, we then find a finite set of places $S\subset M_K$ outside of which we have strong bounds on $V_v$, so that we can show the sum $\sum_{v \in M_K\setminus S} N_v \, V_v(t)$ is uniformly bounded on $X(\Kbar)$.  Combined with the bound $C_v$ for each place $v$, we obtain a uniform bound on the sum $\sum_{v \in M_K} N_v \, V_v(t)$, for all $t\in X(\Kbar)$.  Averaging over Galois orbits will complete the proof of Theorem \ref{dyn VCH}.

\begin{proof}[Proof of Theorem \ref{local dyn VCH}]  
Fix $\gamma \in \cS(F, A)$.  First assume that the pair $(f,a)$ is hole-avoiding at $\gamma$, as defined in \S\ref{hole-avoiding}.  Choose functions $\alpha, \beta \in k$ so that $\ord_\gamma \beta F = \ord_\gamma \alpha A = 0$.  This places us in the setting required for the results of Section \ref{continuity}.  For the function $g_v$ defined there, note that  
\begin{equation} \label{g - V}
	g_v(t) - V_v(t) = \frac{1}{d-1}\log|\beta_t \alpha_t^{d-1}|_v  +  \frac{1}{(2g+1)(d-1)} \ord_\gamma ( \beta \alpha^{d-1}) \log^+|\xi^\gamma(t)|_v
\end{equation}
on an $M_K$-neighborhood of $\gamma$, as a consequence of \eqref{adjustment}, and this difference is continuous at all places $v \in M_K$ and an $M_K$-bounded function.  

If the pair $(f,a)$ fails to be hole-avoiding at $\gamma$, then from Theorem \ref{B point}, we can find a change of coordinates $B \in \mathrm{GL}_2(k)$ and pass to iterates so that the pair $(Bf^nB^{-1}, B(f^m(a)))$ is hole-avoiding at $\gamma$.  From properties \eqref{conjugation} and \eqref{iterate} of the escape rates, we can replace the lifts $(F,A)$ with $(BF^nB^{-1}, BF^m(A))$, and these changes do not affect the computation of $V_v$ on an $M_K$-neighborhood of $\gamma$ (outside of $\gamma$ itself, where the specialization $B_\gamma$ may fail to be invertible), except to multiply it by $d^m$ at every place $v$.  So we can assume that $(f,a)$ is hole-avoiding at $\gamma$.  

We can apply Theorem \ref{uniform convergence} to conclude that $V_v$ is a continuous function on an $M_K$-neighborhood of $\gamma$, for every $v \in M_K$, and that it extends to a continuous function on the closure of this neighborhood in the Berkovich analytification of $X$, for each $v$.  This completes the proof of Theorem \ref{local dyn VCH}, because the continuity of $V_v$ -- when bounded away from the elements of $\cS(F,A)$ in $X_v^{an}$ -- is immediate from the definitions of the escape rates $G_{F_t, v}(A_t)$ and the local height functions for $h_D$.
\end{proof}

\begin{proof}[Proof of Theorem \ref{dyn VCH}]
Fix $\gamma \in \cS(F,A)$.  As in the proof of Theorem \ref{local dyn VCH}, it suffices to assume that $(f,a)$ is hole-avoiding at $\gamma$.  Choose functions $\alpha, \beta \in k$ so that $\ord_\gamma \beta F = \ord_\gamma \alpha A = 0$. Let $S_\gamma$ be a finite set of places of the number field $K$ so that the function in \eqref{g - V} vanishes on an $M_K$-neighborhood of $\gamma$ for all $v \in M_K \setminus S_\gamma$.  The function $V_v$ for the given pair $(F,A)$ then coincides with the function $V_v$ for the pair $(\beta F, \alpha A)$ for all $v \in M_K\setminus S_\gamma$ on an $M_K$-neighborhood of $\gamma$ and is equal to $g_v$ at these places.  

We can enlarge the finite set $S_\gamma$ so that Propositions \ref{good bound on g_v} and \ref{alpha} imply the existence of an $M_K$-quasiconstant $\mathfrak{a}(\gamma)$ for which 
\begin{equation}  \label{sum bound}
	|V_v(t)| = |g_v(t)| \leq  \mathfrak{a}_v(\gamma)
\end{equation}
for all $v \not\in S_\gamma$ and $t$ in an $M_K$-neighborhood of $\gamma$.  

Let $\mathfrak{U}$ be the union of these $M_K$-neighborhoods over all $\gamma \in \cS(F,A)$.  From Proposition \ref{F v bounds}, we know that there exists an $M_K$-constant $\mathfrak{c}$ so that 
\begin{equation} \label{F c}
 e^{-\mathfrak{c}_v}   ~ \leq  ~ \frac{\|F_t(z,w)\|_v}{\|(z,w)\|_v^d} ~ \leq ~  e^{\mathfrak{c}_v}
\end{equation}
for all $t \in X(\bC_v)$ outside of $\mathfrak{U}_v$ and all $v \in M_K$.   From Proposition \ref{global bound}, we can increase the $M_K$-constant $\mathfrak{c}$ so that 
\begin{equation} \label{A c}
	e^{-\mathfrak{c}_v}   ~ \leq  ~ \|A_t\|_v  ~ \leq ~  e^{\mathfrak{c}_v}
\end{equation}
for all $t\in X(\bC_v)\setminus\mathfrak{U}_v$.

Let $S \subset M_K$ be a finite set containing $S_\gamma$ for each $\gamma \in \cS(F,A)$ and containing all places for which $\mathfrak{c}_v \not=0$ and for which $\mathfrak{U}_v$ is not equal to the union $\bigcup_{\gamma \in \cS(F,A)} \{|\xi^\gamma|_v > 1\}$.  
Now fix $t\in X(\Kbar)$. If $t \not\in \mathfrak{U}_v$ at $v \not\in S$, we have 
\begin{align}\label{away}
V_{v}(t) =G_{F_t,v}(A_t) = 0
\end{align}
because $\|F^n_t(A_t)\|_v = 1$ for all $n$, from \eqref{F c} and \eqref{A c}. On the other hand, if $t \in \mathfrak{U}_v$ for $v \not\in S$, then we still have the bound
\begin{align}\label{near}
|V_{v}(t)| \leq \mathfrak{a}_v(\gamma)
\end{align}
from \eqref{sum bound}. Recalling the summability of the bounds in \eqref{sum bound} near each $\gamma \in \cS(F,A)$, inequalities \eqref{away} and \eqref{near} yield
	$$\sum_{v \not\in S} N_v\,  |V_{v}(t)| ~ \leq  ~ C := \sum_{\gamma\in \cS(F,A)} \sum_{v \not\in S} N_v \, \mathfrak{a}_v(\gamma)  < \infty,$$
for each $t\in X(\Kbar)$.

By the continuity of $V_v$ on $X_v^{an}$ for every $v \in M_K$, from Theorem \ref{local dyn VCH}, there is a constant $C_v$ for each $v \in S$ so that $|V_v(t)| \leq C_v$ for all $t \in X(\Kbar)$.  This gives 
	$$\sum_{v \in M_K} N_v \, |V_{v}(t)| ~ \leq  ~ C  + \sum_{v \in S} N_v \, C_v  < \infty$$
for all $t \in X(\Kbar)$.  It follows that, taking averages over the Galois orbit of $t$, we have 
	$$\frac{1}{|\Gal(\Kbar/K)\cdot t|}  \sum_{x \in \Gal(\Kbar/K)\cdot t} \sum_{v \in M_K} N_v \, |V_{v}(t)| ~ \leq  ~ C  + \sum_{v \in S} N_v \, C_v$$
for all $t \in X(\Kbar)$. This completes the proof of Theorem \ref{dyn VCH}.
\end{proof}

\bigskip
\section{Examples}
\label{example section}

In this final section, we present examples to illustrate some of the subtle phenomena that can arise for non-polynomial maps $f: \bP^1 \to \bP^1$, even in the simplest setting of degree $d = 2$, with $K = \bQ$ and $k = \bQ(t)$.  

\subsection{The difference of heights in Theorem \ref{dyn VCH} is bounded but not $M_K$-bounded.}  \label{all primes}  
We first present an example, already seen in Remark \ref{quasi adelic}, where the conditions of Theorem \ref{dyn VCH} hold but the functions $V_v$ of Theorem \ref{local dyn VCH} are nontrivial at infinitely many places $v$ of $K = \bQ$. A mechanism to construct many other such examples appears in \cite{Mavraki:Ye}. Consider
	$$f(z) = \frac{z(z+1)}{z+t}$$ 
defined over $k = \bQ(t)$.  Put 
	$$F(z,w) = (z(z+w), (z+tw)w),$$
so that $\cS(F)$ consists only of the three points $t=0, 1, \infty$ in $X = \bP^1$.  Let $a = 1$, and take $A = (1,1)$ so that $\cS(F,A) = \cS(F) = \{0, 1, \infty\}$.  The point $a$ is totally Fatou as a consequence of Theorem \ref{B point}, because the pair $(f,a)$ is hole-avoiding at all points of $\cS(F)$.  Indeed, at $t=0$, we have $F_0(z,w) = (z(z+w), zw)$ with hole at $z/w=0$ and orbit $f_0^n(a) = n+1$ for all $n\geq 0$.  At $t=1$, we have $F_1(z,w) = (z(z+w), (z+w)w)$ with hole at $z/w = -1$, and orbit $f_1^n(a) = 1$ for all $n\geq 0$.  Finally, at $t=\infty$, we can choose a new lift $F' = \frac1t F$ so that $(F')_\infty(z,w) = (0, w^2)$ with hole at $z/w = \infty$ and orbit $f_\infty^n(a) = 0$ for all $n \geq 1$.   

It follows from these computations that $D = D(F,A) = (\infty)$ is the divisor of degree 1 on $X = \bP^1$ supported at the point $t = \infty$.  This implies, in particular, that $\hat{h}_f(a) = 1$.

Fix a prime $p$ of $\bQ$.  To see that the function $V_p$ is nontrivial on $X$, it suffices to show that $V_p(0) \not=0$.   Let $A_n = F^n(A)$.  As explained in Remark \ref{quasi adelic}, we have $\|(A_n)_0\|_p = 1$ for all $n< p$, and $\|(A_n)_0\|_p < 1$ for all $n\geq p$.  As computed in \eqref{decreasing}, we know that $\|(A_n)_0\|_p^{1/2^n}$ is a decreasing sequence for all primes $p$, so that the $\alpha_p$ of Proposition \ref{alpha} (defined as $g_p(0)$ for the function $g_p(t) = \lim_{n\to\infty} 2^{-n} \log \|(A_n)_t\|_p$ in a $p$-adic neighborhood of $t=0$) is non-zero for all primes $p$.  Moreover, as explained in the proof of Theorem \ref{local dyn VCH}, we also have that $V_p(0) = g_p(0)$ and so $V_p(0) = \alpha_p < 0$ for all primes $p$.

\subsection{All known non-polynomial examples are totally Fatou}
Here we survey the results in the literature where the conclusions of Theorems \ref{dyn VCH} and \ref{local dyn VCH} were known for examples $f: \bP^1 \to \bP^1$ that are {\em not} polynomial maps (nor conjugate to a polynomial).  In every case, the points $a\in \bP^1(k)$ that were treated satisfy our totally Fatou hypothesis. 

The first example is the one presented in the Introduction, where the variation of canonical height $t\mapsto \hat{h}_{f_t}(p_t)$ for a family of Latt\`es maps $f_t$ -- those arising as quotients of endomorphisms of ellptic curves -- is known to differ from a Weil height for a $\bQ$-divisor on the base curve $X$ by a bounded amount, for any choice of $p \in \bP^1(k)$ \cite{Tate:variation}.  The continuity of the local contributions $V_v$, as defined in Theorem \ref{local dyn VCH}, was shown by Silverman in \cite{Silverman:VCHII}.   Also as mentioned in the Introduction, it is well known that all points are totally Fatou for these maps; see, e.g., the computation of the Berkovich Julia set in \cite[\S5]{FRL:ergodic}. Alternatively, note that the existence of a N\'eron model forces all points to be hole-avoiding in appropriate coordinates.

In \cite{GHT:rational}, the authors prove Theorem \ref{dyn VCH} for rational maps $f$ defined over $k = K(X)$ for a curve $X$ and points $c\in \bP^1(k)$, under the assumptions that 
\begin{enumerate}
\item  there exists $t_0 \in X$ so that the map $f$ has good reduction at all $t \not= t_0$;
\item $f$ has a super-attracting fixed point at $z = \infty$; and 
\item the point $c$ satisfies $\ord_{t_0} f^n(c) \to -\infty$.
\end{enumerate}
Condition (3) implies that $c$ is in the basin of attraction of the super-attracting fixed point at $\infty$, so it is clearly Fatou at $t_0$.  (The hypothesis (3) is stated in \cite[Theorem 5.4]{GHT:rational} as $\{\deg f^n(c): n\geq 0\}$ is unbounded, but for a notion of degree defined in their Section 5 on the regular functions on $X\setminus\{t_0\}$ and extended to $k$ after equation (5.4).)

In \cite{Ghioca:Mavraki:variation}, the authors studied maps of the form 
	$$f(z) = \frac{z^d + t}{z}$$
over $k = \Qbar(t)$, for $d \geq 3$, and they prove Theorem \ref{dyn VCH} for all points $a \in \bP^1(k)$.  (The map $f$ for $d = 2$ is isotrivial, making the theorem true but much easier.)  In this example, the point $z=\infty$ is a super-attracting fixed point, and there are two places of bad reduction, at $t = 0$ and $t=\infty$.  All points $a \in \bP^1(k)$ are totally Fatou.  Indeed, at $t=0$, the reduction is $f_0(z) = z^{d-1}$ with only hole at $z=0$.  So the only points we need to consider are those which vanish at $t=0$. But for any integer $m \geq 1$, if $\ord_0 a = m$, then $\ord_0 f(a) = 1-m \leq 0$, so $f(a)$ which will no longer specialize to 0 at $t=0$; this implies that the pair $(f, f(a))$ is hole-avoiding at $t=0$ for all $a \in \bP^1(k)$.  At $t=\infty$, a computation shows that if $\ord_\infty a = r <  0$, then $\ord_\infty f(a) = (d-1) r$; iterating implies that $f^n(a) \to \infty$ in the $\infty$-adic topology, so the point $a$ will be Fatou at $t=\infty$.  Moreover, if $\ord_\infty a = r \geq 0$, then $\ord_\infty f(a) = -r-1 < 0$, and again $a$ is Fatou at $\infty$.

In \cite{DWY:QPer1}, the authors consider
	\begin{align}\label{per1}
	f(z) = \frac{\lambda z}{z^2 + t z + 1}
	\end{align}
for a fixed $\lambda \not=0$ in $\Qbar$, defined over $k = \Qbar(t)$, having a fixed point of multiplier $\lambda$.  For $\lambda$ not a root of unity or for $\lambda=1$, the result of Theorem \ref{dyn VCH} is obtained there for the critical points $c_\pm = \pm 1$ of $f$.  The critical points will be totally Fatou for any choice of $\lambda$.  It suffices to check the dynamics of $f$ at $t=\infty$.  For $\lambda = 1$, we can conjugate $f$ by $B(z) = 1/(tz)$ so the new map $z + 1 + 1/(t^2z)$ specializes to $z \mapsto z+1$ with hole at $z=0$, and the critical values in the new coordinate system $B(f(c_\pm)) = (\pm 2+t)/t$  specialize to $z=1$, so the pairs $(f, f(c_\pm))$ are seen to be hole-avoiding in the new coordinate system.  For $\lambda$ not a root of unity, the map $f$ can be conjugated to a map that specializes to $z \mapsto \lambda \, z$ with a hole at $z = 1$, and so that the critical values $f(c_\pm)$ in the new coordinates will specialize to $z = \lambda$.  Again the pairs $(f, f(c_\pm))$ are hole-avoiding in the new coordinate system.  These facts appear in the proof of \cite[Proposition 2.2]{DWY:QPer1} and in \cite[\S5]{D:moduli2}; these cases are also covered by \cite[Lemma 3.4]{Kiwi:quad}.  Finally, in \cite{Mavraki:Ye}, the authors obtain the result of Theorem \ref{dyn VCH} for the maps $f$ of the form \eqref{per1} when $\lambda$ is a root of unity and for a large class of points $c\in\Qbar(t)$ satisfying a hole-avoiding condition at the place of bad reduction $t=\infty$. This includes in particular maps of the form \eqref{per1} with $c_\pm=\pm1$. 
As explained in Theorem \ref{B point} above, this means that the points considered are totally Fatou.

\subsection{Julia points:  irrational local heights and $\R$-divisors}  \label{Julia example 1}
This next example is a map of degree 2 defined over the field $k = \bQ(t)$, with the property that all points with infinite orbit that lie in a local non-archimedean Julia set at the place $t=0$ of $k$ will have an irrational local canonical height.  If such a point can be algebraic over $k$, it would show that the conclusion of Theorem \ref{dyn VCH} would fail for Julia points, as stated; the divisor $D$ should be an $\R$-divisor on the curve $X$.  Set
\begin{equation} \label{Julia 1 f}
	f(z) = \frac{(t^2+t+1)z^2 + tz+ t^2-1}{(2t^2+t)z + t}.
\end{equation}
This map has fixed points at $z=1$ (with multiplier $1/t$) and at $z=-1$ (with multiplier $1/t^2$), and at $z=\infty$.  At the place $t=0$ of $k$, these fixed points at $\pm 1$ are repelling, and the fixed point at $\infty$ is attracting (but not super-attracting).  The Julia set in $\bP^{1,an}_{t=0}$, defined over the field $\bL$ of formal Puiseux series in $t$, is a Cantor set of Type I points, and $f$ is conjugate to the full 2-shift \cite[Theorem 3(1)]{Kiwi:quad}.  All points outside of the Julia set will tend to $\infty$ under iteration.  This map $f$ exhibits a polynomial-like behavior near its Julia set, and it can be computed that the Julia set is a subset the formal completion $k_0 := \bQ[[t]]$.  But this example is not {\em strongly} polynomial-like, in the sense of \cite[Theorem 1.5]{DG:rationality}, because the multipliers at the two repelling fixed points have distinct absolute values.

We can compute local canonical heights over the field $k$ with the procedure described in \cite{DG:rationality}.  
In homogeneous coordinates, put
	$$F(z,w) = ((t^2+t+1)z^2 + tzw + (t^2 - 1) w^2, ~ (2 t^2 + t) z w + t w^2).$$
At $t=0$, we have $F_0(z,w) = (z^2-w^2, ~0)$, and the Julia set is contained in the hole-directions $\pm 1$ from the Gauss point $\zeta_G$, i.e., in the union of the two disks $D_\pm = \{z \in \bL:  |z - (\pm 1)|_0 < 1\}$. 
The conjugacy between $f$ on its Julia set and the shift map on 2 symbols is given by the itinerary of a point as it moves between $D_+$ and $D_-$.  For a point $a \in k_0$ with lift $A \in (k_0)^2\setminus\{(0,0)\}$, a sequence of orders is defined by 
\begin{equation} \label{tau order}
	\tau_n := \ord_{t=0} F^n(A) = 2 \, \tau_{n-1} + \sigma_{n-1}
\end{equation}
so that 
	$$G_{F,0}(A) = -\lim_{n\to\infty} \frac{\tau_n}{2^n} = - \tau_0 - \sum_{n=1}^\infty \frac{\sigma_{n-1}}{2^n}.$$
From the formula for $F$, we can compute that $\sigma_n = 1$ if $f^n(a) \in D_+$ and $\sigma_n = 2$ if $f^n(a) \in D_-$, for all $n\geq 0$.  Because of the conjugation to the shift map, we see that the sequence $\{\sigma_n\}$ is eventually periodic if and only if the point $a$ is eventually periodic.  Therefore, $G_{F,0}(A)$ (and so also any presentation of the geometric local canonical height $\hat{\lambda}_{f,0}(a)$ at $t=0$) is irrational for all Julia points with infinite orbit.

\begin{remark}
The function $f$ of \eqref{Julia 1 f} is conjugate to $z \mapsto \frac{t^2 z^2+z}{t z + t^2}$, in a standard normal form for quadratic rational maps, with fixed points at 0 and $\infty$ of specified multipliers (in this case, having multiplier $1/t^2$ at 0 and $1/t$ at $\infty$).  We then moved the two repelling fixed points to $1$ and $-1$ and the attracting fixed point to $\infty$.  
\end{remark}

\subsection{Julia points with divergent escape rates}  \label{divergent example}
Our final example is
\begin{equation} \label{Julia 2 f}
	f(z) = \frac{z^2 + (t^2-t-1)z  - t^3 - 2t^2 + t}{z - t^2 - 1},
\end{equation}
defined over the field $k = \bQ(t)$, at the place $\gamma$ corresponding to $t=0$.  As for the example \eqref{Julia 1 f}, all Julia points at $t=0$ with infinite orbit for $f$ will have an irrational local canonical height at $t=0$.  This can be seen from the proof of \cite[Theorem 1.3]{DG:rationality}, because this $f$ is conjugate to the map $z \mapsto \frac{(z+1)(z-t)}{z+t}$ studied there, combined with an identification of the Julia set with the shift on 2 symbols \cite[Proposition 4.2]{Kiwi:quad}.

We construct (formal) points $a \in \bQ[[t]]$ in the Julia set of $f$ at $t=0$ so that the sequence of functions \eqref{g_n} that define $V_\infty$ (at the archimedean place) will diverge at $t=0$.  We do not know if the points $a$ we construct can be algebraic over $k$, nor even if the series will converge on a disk around $t=0$.  We use these examples to illustrate some of the features of Julia points that do not arise for the Fatou points.  

\begin{remark} \label{non compact}
As we shall see, taking any unbounded sequence of positive integers $\{m_k\}_{k \geq 0}$ in the construction below, this example also shows that the orbits of points in $\bP^1(k_\gamma)$ can have non-locally-compact closures.  This is distinct from what happens for polynomials; compare \cite[Theorem 3]{Favre:Gauthier:continuity}.
\end{remark}

More precisely, we construct examples so that the sequence 
	$$\alpha_n := \frac{1}{2^n} \log\|(A_n)_0\|_v,$$
as defined and studied in Proposition \ref{alpha}, will diverge to $-\infty$ at the place $v = \infty$.  In particular, this would show that -- if the point $a$ defines a convergent series in $\bQ[[t]]$ -- the conclusion of Theorem \ref{uniform convergence} would fail.  That is, the sequence of functions
	$$g_n(t) := \frac{1}{d^n} \log\|(A_n)_t\|$$
would converge, locally uniformly on a punctured disk around $t=0$, and we know that the limit function $g(t)$ must be bounded by $o(\log|t|)$ as $t\to 0$ \cite[Proposition 3.1]{D:stableheight}.  But the convergence to $g$ would not be uniform in a neighborhood of $t=0$.

For the construction, note first that $f$ specializes to the identity transformation $f_0(z) = z$ at $t=0$ with a hole at $z=1$.  The Berkovich Julia set is contained in the direction $z = 1$ from the Gauss point, and all Julia points have the form $1 + m \, t + O(t^2)$ for some integer $m \geq 0$.  We have 
	$$f(1 + w \, t + O(t^2)) = 1 + (w-1) \, t + O(t^2)$$
for all $w \not= 0$, and 
	$$f(1 + w\, t^2 + O(t^3)) = 1 + \frac{1+w}{1-w}\, t + O(t^2)$$
for all $w \not= 1$.  

For any sequence of positive integers $\{m_k\}_{k \geq 0}$, there is a unique point $a \in \bQ[[t]]$ so that 
	$$a = 1 + m_0 \, t + O(t^2)$$
and 
	$$f^{m_0+\cdots+m_{k-1} + k}(a) = 1 + m_k \, t + O(t^2)$$
for all $k\geq 1$ \cite[Theorem 1.3]{DG:rationality}, \cite[Proposition 4.2]{Kiwi:quad}. Let
	$$F(z,w) =   \left( z^2 + (t^2-t-1)zw  +(- t^3 - 2t^2 + t)w^2, \, zw - (t^2 + 1)w^2\right)$$
be a lift of $f$.  Set 
	$$A = (a, 1) = (1 + m_0 \, t + O(t^2), \, 1)$$
for the $a$ associated to a given sequence $\{m_k\}_{k\geq 0}$. For each $n\geq 1$, we set
	$$A_n = t^{-\sigma_{n-1}} F(A_{n-1}),$$
where $\sigma_{n-1}$ is chosen so that $\ord_{t=0} A_n = 0$, as above in \eqref{tau order}.  In fact, $\sigma_n = 1$ whenever $f^n(a) = 1 + w t + O(t^2)$ for $w \not=0$, and $\sigma_n = 2$ for $f^n(a) = 1 + O(t^2)$.  For each $n \geq 1$, we let $(A_n)_0$ denote the specialization at $t=0$, and set 
	$$\alpha_n = \frac{1}{2^n} \log\|(A_n)_0\|$$
in the archimedean norm.   

We now show that, by choosing the sequence $\{m_k\}_{k\geq 0}$ to grow to infinity sufficiently fast, we can have $\liminf_{n\to\infty} \alpha_n = -\infty$. For a lift $(y_0 + x_1 t + O(t^2), \, y_0 + y_1 t + O(t^2))$ of the point $1 + m\, t + O(t^2)$, with $m > 0$, we have $(x_1-y_1)/y_0 = m$, and this gives 
\begin{multline*}
F \left( y_0 + x_1 t + O(t^2), \, y_0 + y_1 t + O(t^2)\right) = t \, \left( y_0 (x_1-y_1), \, y_0 (x_1-y_1)\right) + O(t^2) \\ =  t \, \left( m\, y_0^2 , \, m\, y_0^2 )\right) + O(t^2)
\end{multline*}
for $m\not= 0$.  Now suppose we take any lift $(y_0 + y_1 t + x_2 t^2 + O(t^3), \, y_0 + y_1 t + y_2 t^2 + O(t^3))$ with $y_0 \not= 0$ of a point $p = 1 + c\, t^2 + O(t^3)$, with $c = \frac{m-1}{m+1}$ for $m \geq 1$ so that $f(p) = 1 + m\, t + O(t^2)$.  Then we must have $\frac{x_2-y_2}{y_0} = \frac{m-1}{m+1}$, and this gives
\begin{multline*}
F \left(y_0 + y_1 t + x_2 t^2 + O(t^3), \, y_0 + y_1 t + y_2 t^2 + O(t^3)\right) \\
 = t^2 \, \left( y_0 (x_2 -y_0 - y_2), \, y_0 (x_2 -y_0 - y_2)\right) + O(t^3) \\ =  t^2 \, \left( \frac{-2 y_0^2}{m+1} , \, \frac{-2 y_0^2}{m+1} \right) + O(t^3).
\end{multline*}
Iterating the point $A$, we find that 
	$$(A_1)_0 = (m_0, \, m_0)$$
	$$(A_2)_0 = \left(m_0^2(m_0-1), \, m_0^2(m_0-1)\right)$$
all the way to 
	$$(A_{m_0})_0 = \left( m_0^{2^{m_0-1}} (m_0-1)^{2^{m_0-2}} \cdots 1, \, m_0^{2^{m_0-1}} (m_0-1)^{2^{m_0-2}} \cdots 1\right) =:  (Y_{m_0}, Y_{m_0})$$
so that 
	$$\alpha_{m_0} = \frac{1}{2^{m_0}} \log|Y_{m_0}| = \sum_{j=1}^{m_0} \frac{1}{2^j} \log(m_0-j+1).$$
The point $A_{m_0}$ is a lift of $f^{m_0}(a)= 1 + 0\,t + \frac{m_1-1}{m_1+1} t^2 + O(t^3)$, so that 
	$$(A_{m_0+1})_0 = \left( \frac{-2}{m_1+1} Y_{m_0}^2, \, \frac{-2}{m_1+1} Y_{m_0}^2\right). $$
Thus,
	$$\alpha_{m_0+1} = \alpha_{m_0} + \frac{1}{2^{m_0+1}} \log\left( \frac{2}{m_1+1} \right).$$
Note that this $\alpha_{m_0+1}$ can be made as negative as desired by choosing $m_1 \gg m_0$.  Continuing to iterate, we have 
	$$\alpha_{m_0 + 1 + s} = \alpha_{m_0+1} + \sum_{j=1}^s \frac{1}{2^{m_0+1+j}} \log(m_1 - j + 1)$$
for all $s = 1, \ldots, m_1$.  Then
	$$\alpha_{m_0 + m_1 + 2} = \alpha_{m_0 + m_1 + 1} + \frac{1}{2^{m_0 + m_1+2}} \log\left( \frac{2}{m_2+1} \right).$$
Again, we can make $\alpha_{m_0 + m_1 + 2}$ as negative as desired by choosing $m_2 \gg m_1$.  We see that the pattern continues, and so, by choosing the sequence $\{m_k\}_{k\geq 0}$ to grow to infinity very fast, we conclude that the sequence $\{\alpha_n\}_{n\geq 0}$ is unbounded from below.



\bigskip \bigskip

\def\cprime{$'$}

\end{document}